\documentclass[11pt]{amsart}
\usepackage{geometry}
\usepackage{float}
\usepackage[usenames]{color}
\usepackage{graphicx}
\usepackage{amscd,amssymb,verbatim,xcolor}
\usepackage{longtable, multirow}
\usepackage{soul,cancel}
\usepackage{tikz}
\usepackage{amsmath}
\usepackage{ragged2e}

\date{\today}

\newtheorem{theorem}{Theorem}[section]
\newtheorem{algorithm}[theorem]{Algorithm}

\newtheorem{corollary}[theorem]{Corollary}

\newtheorem{definition}[theorem]{Definition}
\newtheorem{example}[theorem]{Example}
\newtheorem{lemma}[theorem]{Lemma}
\newtheorem{proposition}[theorem]{Proposition}
\newtheorem{remark}[theorem]{Remark}

\begin{document}
\title[Dirac series of $\mathrm{GL}(n)$ over an Archimedean field]{Dirac series of $\mathrm{GL}(n)$ over an Archimedean field}
\author{Yihao Ding}
\author{Hongfeng Zhang}

\address[Ding]{School of Mathematics, Suzhou University of Science and Technology, Suzhou 215009, China}
\email{dyhqxy0506@163.com}

\address[Zhang]{School of Science and Engineering, The Chinese University of Hong Kong, Shenzhen,
Guangdong 518172, P. R. China}
\email{zhanghongf@pku.edu.cn}

\begin{abstract}
Motivated by the $(\mathfrak{g},K)$-cohomology and Dirac cohomology, we determine  Dirac series of $\mathrm{GL}(n,\mathbb{H})$, and show that the spin lowest $K$-type of any Dirac series, which determines the Dirac cohomology, is unique and multiplicity-free for both $\mathrm{GL}(n,\mathbb{H})$ and $\mathrm{GL}(n,\mathbb{R})$. This verifies a conjecture about uniqueness of the spin lowest $K$-type of Dirac series for $\mathrm{GL}(n,\mathbb{R})$ proposed by Dong and Wong \cite{Dong_Wong_2:2022}.
\end{abstract}

\subjclass[2010]{Primary 22E46.}

\keywords{Dirac cohomology, unitary representation, spin lowest $K$-type.}

\maketitle
\setcounter{tocdepth}{1}

\section{Introduction}
In the 1970s, Parthasarathy \cite{Parthasarathy:1972} used the Dirac operator to construct discrete series of real reductive groups and introduced an inequality, known as Dirac inequality, as a necessary condition for a representation to be unitarizable. In the 1980s, Vogan \cite{Vogan:1986} applied Dirac inequality on the bottom layer $K$-types as a key ingredient to obtain the unitary dual of the general linear group over Archimedean fields, and Barbasch \cite{Barbasch:1989} extended these ideas to  the classical complex groups.  Later in the 1990s, Vogan defined the Dirac cohomology of representations of reductive groups, consisting of certain  $\widetilde{K}$-types, and conjectured that the highest weights of these $\widetilde{K}$-types could tell us the infinitesimal characters of the representations. This conjecture implies that Dirac cohomology is a refinement of Dirac inequality and was proved by Huang-Pand\v{z}i\'c \cite{Huang_Pandzic:2002}.

Dirac series of $G$, denoted by $\widehat{G}^d$, is the set of equivalence classes of irreducible unitary representations with nonzero Dirac cohomology. Among the entire unitary dual, Dirac series are precisely the ones where Parthasarathy's Dirac inequality becomes an equality on some $K$-types. There are lots of important unitary representations contained in Dirac series, such as discrete series, most of $A_{\mathfrak{q}}(\lambda)$ modules, unitary highest weight modules and so on. In view of the research by Barbasch and Pand\v{z}i\'c \cite{Barbasch_Pandzic:2019}, Dirac series should have applications in the theory of automorphic forms as well. Moreover, the problem of classifying $\widehat{G}^d$ itself is also interesting and quite subtle.

The classification of Dirac series is known for the following groups.
\begin{itemize}
\item[(i)] Complex cases: complex classical groups \cite{Barbasch_Dong_Wong:2022}, complex $G_2$, $F_4$, $E_6$, $E_7$, $E_8$  \cite{Dong_Wong_3:2022}.
\item[(ii)] Real cases: $\mathrm{GL}(n,\mathbb{R})$  \cite{Dong_Wong_2:2022}, $\mathrm{U}(n,2)$, $E_{6(2)}$,
$E_{6(6)}$, $E_{6(-26)}$,
$E_{6(-14)}$, $E_{7(7)}$, $E_{7(-25)}$, $E_{7(-5)}$, $E_{8(-24)}$, $F_{4(4)}$, $F_{4(-20)}$, $G_{2(2)}$ \cite{Ding_Dong_Wei:2023}.
\end{itemize}

To compute the Dirac cohomology more effectively,
Dong \cite{Dong:2013} introduced spin norm and spin lowest $K$-type. For Dirac series, the spin lowest $K$-types are exactly the $K$-types contributing to the Dirac cohomology. In the list of the groups whose Dirac series have been classified, for the complex cases, each Dirac series has a unique spin lowest $K$-type with multiplicity one; while in the real cases, Dirac series contain several spin lowest $K$-types, each occurring with multiplicity one, except for $E_{7(7)}$. $E_{7(7)}$ has 125 fully supported Dirac series and two among them have unique spin lowest $K$-type with multiplicity two. 

Note that from spin lowest $K$-types of Dirac series, one can get the Dirac cohomology immediately. Let us show how the uniqueness and multiplicity of spin lowest $K$-types might be important for automorphic forms. By \cite{Huang_Pandzic:2006}, Dirac cohomology has a connection with $(\mathfrak{g},K)$-cohomology, and hence automorphic forms:
Given a unitary representation $\pi$ and a finite dimensional module $F$ of $G$, the $(\mathfrak{g},K)$-cohomology with coefficients in $F$ is connected with Dirac cohomology as follows,
\[\mathop{\bigoplus}\limits_{i\in \mathbb{N}}H^i(\mathfrak{g},K,\pi\otimes F^*)\cong \mathrm{Hom}_{\widetilde{K}}\big(H_D(F),H_D(\pi)\big),\]
where $F^*$ is the dual of $F$ and $H_D(\bullet)$ is the Dirac cohomology of $\bullet$.
Therefore, the Dirac cohomology of $\pi$ certainly helps us understand the $(\mathfrak{g},K)$-cohomology. 
Although the irreducible unitary representation with non-zero $(\mathfrak{g},K)$-cohomology has been classified, which are the $A_{\mathfrak{q}}(\lambda)$ with regular integral infinitesimal characters, and they are also exactly the Dirac series with regular integral infinitesimal characters \cite{Huang_Pandzic:2006}. However, it is interesting that there are many Dirac series with half-integral infinitesimal characters and so zero $(\mathfrak{g},K)$-cohomology. It might be useful to understand such Dirac series for automorphic forms.

In this paper, we first classify Dirac series of $\mathrm{GL}(n,\mathbb{H})$. They are proved to contain unique spin lowest $K$-type with multiplicity one, and such $K$-type can be obtained by the algorithm in the $\mathrm{GL}(n,\mathbb{R})$ case. Then we turn to the $\mathrm{GL}(n,\mathbb{R})$ case: Based on the classification of Dirac series of $\mathrm{GL}(n,\mathbb{R})$ \cite{Dong_Wong_2:2022}, we show that any Dirac series of $\mathrm{GL}(n,\mathbb{R})$ also contains unique spin lowest $K$-type with multiplicity one. This verifies Conjecture 4.5 of \cite{Dong_Wong_2:2022}.
As above, the multiplicity-free results about the spin lowest $K$-types of Dirac series of $\mathrm{GL}(n,\mathbb{H})$ and $\mathrm{GL}(n,\mathbb{R})$ should be interesting for the theory of automorphic forms.

In summary, combining the results of \cite{Barbasch_Dong_Wong:2022} and \cite{Dong_Wong_2:2022}, we complete the classification of the Dirac series of $\mathrm{GL}(n,\mathbb{F})$, $\mathbb{F} = \mathbb{R},\mathbb{C}$ or $\mathbb{H}$, and conclude that any Dirac series of $\mathrm{GL}(n,\mathbb{F})$ contains a unique spin lowest $K$-type with multiplicity one. From the spin lowest $K$-type, we can get the Dirac cohomology of Dirac series of $\mathrm{GL}(n,\mathbb{F})$.

Following our results, one would like to conjecture that any Dirac series of $\widetilde{\mathrm{GL}}(n,\mathbb{R})$, the universal covering of $\mathrm{GL}(n,\mathbb{R})$, contains a unique spin lowest $K$-type with multiplicity one. We will continue to study this case in the future.    
\section{Preliminaries}

\subsection{Dirac cohomology}

We begin with the basic definition of Dirac cohomology, for more details, see \cite{Huang_Pandzic:2006}.
Let $G$ be a real reductive group with a maximal compact subgroup $K$. Let $\mathfrak{g}_0=\mathfrak{k}_0\oplus \mathfrak{p}_0$ be the corresponding Cartan decomposition. Here $\mathfrak{g}_0$, $\mathfrak{k}_0$ denote the Lie algebras of $G$ and $K$  respectively. Let $\mathfrak{t}_0$ be a maximal abelian subalgebra of $\mathfrak{k}_0$. Let $\mathfrak{a}_0$=$Z_{\mathfrak{p}_0}(\mathfrak{t}_0)$, the centralizer of $\mathfrak{t}_0$ in $\mathfrak{p}_0$. Then $\mathfrak{h}=\mathfrak{t}\oplus\mathfrak{a}$ is a fundamental Cartan subalgebra of $\mathfrak{g}$. Here we denote the corresponding complexified Lie algebra by dropping the subscript.

We denote by $\Delta(\mathfrak{g},\mathfrak{h})$ the root system of $\mathfrak{g}$ with respect to $\mathfrak{h}$ and denote by $\Delta(\mathfrak{k},\mathfrak{t})$ the root system of $\mathfrak{k}$ with respect to $\mathfrak{t}$. Fix a positive root system $\Delta^{+}(\mathfrak{k},\mathfrak{t})$, and denote the half sums of roots in it by $\rho_c$.

Let $\langle \cdot, \cdot \rangle$ be a nondegenerate symmetric invariant bilinear form on $\mathfrak{g}_0$, which is positive definite on $\mathfrak{p}_0$ and negative definite on $\mathfrak{k}_0$. Fix an orthonormal basis $Z_1,\dots,Z_n$ of $\mathfrak{p}_0$ with respect to the inner product $\langle \cdot, \cdot \rangle|_{\mathfrak{p}_0}$. Let $U(\mathfrak{g})$ be the universal enveloping algebra of $\mathfrak{g}$, and put $C(\mathfrak{p})$ as the Clifford algebra over $\mathfrak{p}$.
The Dirac operator $D \in U(\mathfrak{g}) \otimes C(\mathfrak{p})$ is defined as $$D=\sum_{i=1}^{n}\, Z_i \otimes Z_i.$$
The operator $D$ does not depend on the choice of an orthonormal basis $Z_i$ of $\mathfrak{p}_0$. Furthermore, it is $K$-invariant for the diagonal
action of $K$ given by adjoint actions on both factors.

Let $\widetilde{K}$ be the subgroup of $K\times \mathrm{Spin}\,\mathfrak{p}_0$ consisting of all pairs $(k, s)$ such that $\text{Ad}(k)=p(s)$, where $\text{Ad}: K\rightarrow {\mathrm{SO}}(\mathfrak{p}_0)$ is the adjoint action, and $p: \mathrm{Spin}\,\mathfrak{p}_0\rightarrow {\mathrm{SO}}(\mathfrak{p}_0)$ is the spin double covering map. Here $\mathrm{SO}(\mathfrak{p}_0)$ is defined with respect to $\langle \cdot, \cdot \rangle$ on $\mathfrak{p}_0$.

If $\pi$ is a $(\mathfrak{g}, K)$ module, and if $S_G$ denotes a spin module for
$C(\mathfrak{p})$, then $\pi \otimes S_G$ is a $(U(\mathfrak{g})\otimes C(\mathfrak{p}),
\widetilde{K})$ module, where $\widetilde{K}$ acts on $\pi$ through $K$ and  acts on $S_G$ through $\mathrm{Spin}\,\mathfrak{p}_0$ and the action of $U(\mathfrak{g})\otimes C(\mathfrak{p})$ is
the obvious one. In particular, the Dirac operator $D$ acts on $\pi \otimes S_G$. Vogan defined the Dirac cohomology of $\pi$ as
$$
H_D(X)=\mathrm{ker} D/(\mathrm{ker} D \cap \mathrm{im} D).
$$
When $\pi$ is unitary, $\mathrm{ker} D$ and $\mathrm{im} D$ intersect trivially and the Dirac cohomology of $\pi$ is simply $\mathrm{ker} D=\mathrm{ker} D^2$.

We conclude this subsection by stating the main result of \cite{Huang_Pandzic:2002}.
\begin{theorem}\label{thm-HP}
Let $\pi$ be an irreducible $(\mathfrak{g}, K)$-module with infinitesimal character $\Lambda\in\mathfrak{h}^*$. If $H_D(\pi)$ is non-zero, let $\widetilde{\gamma}$ be any highest weight of any $\widetilde{K}$-type occurring in $H_D(\pi)$. Then there exists $w\in W(\mathfrak{g}, \mathfrak{h})$ such that
\begin{equation}\label{inf-char-HP}
w \Lambda|_{\mathfrak{a}} = 0, \quad w \Lambda|_{\mathfrak{t}} = \widetilde{\gamma} + \rho_c.
\end{equation}

Assume that $\pi$ is unitarizable. Let $\widetilde{\gamma}$ be any highest weight  of any $\widetilde{K}$-type of $\pi \otimes S_G$. By Parthasarathy’s Dirac inequality, we have
\begin{equation}\label{Dirac_ineq}
\| \widetilde{\gamma}+\rho_c\| \geq  \|\Lambda\|.\end{equation}
Then the equality in \eqref{Dirac_ineq} holds if and only if there exists $w\in W(\mathfrak{g},\mathfrak{h})$ such that \eqref{inf-char-HP} holds.
\end{theorem}

\subsection{Spin lowest $K$-types}

\subsubsection{Decomposition of $S_G$}

We fix positive root systems $\Delta^+(\mathfrak{g},\mathfrak{t})$ and $\Delta^+(\mathfrak{k},\mathfrak{t})$ such that $\Delta^+(\mathfrak{g},\mathfrak{t})$ contains $\Delta^+(\mathfrak{k},\mathfrak{t})$. Denote by $\mathfrak{C}_{\mathfrak{g}}$ (resp. $\mathfrak{C}_{\mathfrak{k}}$) the closed dominant Weyl chamber for $\Delta^+(\mathfrak{g},\mathfrak{t})$(resp. $\Delta^+(\mathfrak{k},\mathfrak{t})$). Define
$$
W(\mathfrak{g},\mathfrak{t})^1=\{w \in W(\mathfrak{g},\mathfrak{t})|w(\mathfrak{C}_{\mathfrak{g}}) \subseteq \mathfrak{C}_{\mathfrak{k}}\}.
$$
By a result of Kostant \cite{Kostant:1961}, the multiplication map gives a bijection from $W(\mathfrak{g},\mathfrak{t})^1 \times W(\mathfrak{k},\mathfrak{t})$ onto $W(\mathfrak{g},\mathfrak{t})$. Then $$\{w\Delta^+(\mathfrak{p},\mathfrak{t})|w \in W(\mathfrak{g},\mathfrak{t})^1\}$$
are exactly all the choices of positive root systems of $\Delta(\mathfrak{p},\mathfrak{t})$ that are compatible with $\Delta^+(\mathfrak{k},\mathfrak{t})$.

Let us enumerate the elements of $W(\mathfrak{g},\mathfrak{t})^1$ as $w^0=e,w^1, \dots, w^{s-1}$. For $0 \leq j \leq s-1$, put
\[(\Delta^+)^{(j)}(\mathfrak{p},\mathfrak{t}):=w^j\Delta^+(\mathfrak{p},\mathfrak{t})\ \text{and}\ (\Delta^+)^{(j)}(\mathfrak{g},\mathfrak{t}):=w^j\Delta^+(\mathfrak{g},\mathfrak{t})=w^j\Delta^+(\mathfrak{p},\mathfrak{t}) \cup \Delta^+(\mathfrak{k},\mathfrak{t}).\]
Denote by $\rho_n^{(j)}$(resp. $\rho^{(j)}$) the half sum of roots in $(\Delta^+)^{(j)}(\mathfrak{p},\mathfrak{t})$(resp. $(\Delta^+)^{(j)}(\mathfrak{g},\mathfrak{t})$).

For any dominant integral weight $\lambda\in i\mathfrak{t}^*$, let $E_{\lambda}$ denote the irreducible $\mathfrak{k}$-module with highest weight $\lambda$.

The $\mathfrak{k}$-structure of $S_G$ is given by the following:

\begin{lemma}[\cite{Wallach:1988}, Lemma 9.3.2]\label{Spin decomposition}
Let $l_0=\dim \mathfrak{a}$, then as $\mathfrak{k}$-modules,
$$
S_G \cong 2^{[l_0/2]}\mathop{\bigoplus}_{w \in W(\mathfrak{g},\mathfrak{t})^1}E_{w\rho-\rho_c}
=2^{[l_0/2]}\mathop{\bigoplus}_{0\leq j\leq s-1}E_{\rho^{(j)}_n}
$$

\end{lemma}

\subsubsection{PRV component}

The following theorem is useful to compute $H_D(\pi)$. It combined with \eqref{Dirac_ineq} tells that any $\widetilde{K}$-type of $H_D(\pi)$ must be a PRV-component of $\pi\otimes S_G$. Let $w_0$ be the longest element of $W(\mathfrak{k},\mathfrak{t})$.

\begin{theorem}[\cite{Parthasarathy_Rao_Varadarajan:1967}, Corollaries 1 and 2 to Theorem 2.1]\label{PRV}

Let $\mu$, $\nu$ be two dominant integral weights. Then $E_{
\{\mu+w_0\nu\}}$ occurs with multiplicity one in $E_{\mu} \otimes E_{\nu}$. Here the braces denote the unique dominant weight to which the indicated weight is conjugate. Moreover, if $E_{\tau}$ is any irreducible component of $E_{\mu} \otimes E_{\nu}$, then $\{\mu+w_0\nu\}$ must be a weight of $E_{\tau}$. In particular, we have
$$
||\{\mu+w_0\nu\}+\rho_c|| \leq ||\tau+\rho_c||
$$ and the equality happens if and only if $\tau=\{\mu+w_0\nu\}$.

The irreducible constituent $E_{\{\mu+w_0\nu\}}$ is called the \textbf{PRV component} of $E_{\mu} \otimes E_{\nu}$.

\end{theorem}

\subsubsection{Spin lowest $K$-type}

To compute the Dirac cohomology of a unitary $(\mathfrak{g},K)$-module $\pi$ effectively, \cite{Dong:2013} introduces and considers \emph{the spin lowest $K$-type} of $\pi$.

The spin norm of a $\mathfrak{k}$-type $E_\mu$ is defined to be
\[||\mu||_{spin}:=\min_{\gamma\in\ \{\text{the highest weights in}\ E_{\mu}\otimes S_G\} } ||\gamma+\rho_c||.\]
By the properties of PRV-component and the decomposition of $S_G$,
\[||\mu||_{spin}=\min_{0 \leq j \leq s-1}||\{\mu-\rho_n^{(j)}\}+\rho_c||.\]
Here $\{\mu-\rho_n^{(j)}\}$ denotes the unique dominant weight to which $\mu-\rho_n^{(j)}$ is conjugate under the action of $W(\mathfrak{k},\mathfrak{t})$.
A $K$-type $\mu$ of $\pi$ is called a \textbf{spin lowest $K$-type} of $\pi$ if its spin norm attains the minimum among all the $K$-types of $\pi$.

With the definition of spin lowest $K$-type and the properties of PRV component at hand, Theorem \ref{thm-HP} can be rephrased as follows.

\begin{corollary} \label{cor-HP}
Let $\pi$ be a unitary $(\mathfrak{g},K)$-module with infinitesimal character $\Lambda \in \mathfrak{h}^*$. Then the following are equivalent to each other.

\begin{itemize}
\item[(i)] $H_D(\pi)$ is non-zero.
\item[(ii)] There exists a $w\in W(\mathfrak{g}, \mathfrak{h})$, and a $K$-type in $\pi$ with highest weight $\tau$ such that
\begin{equation}\label{eq-HP2}
w\Lambda|_{\mathfrak{a}} = 0, \quad \{\tau - \rho_n^{(j)}\} = w\Lambda|_{\mathfrak{t}} - \rho_c.
\end{equation}
\item[(iii)] The highest weight $\tau$ of any spin lowest $K$-type of $\pi$ satisfies  $\|\tau\|_{spin}=\| \Lambda\|$.
\end{itemize}

Moreover, if $H_D(\pi)\neq 0$, then
\[H_D(\pi)\cong \mathop{\bigoplus}\limits_{\tiny \begin{array}{c}\tau\in \{\text{spin lowest}\ K\text{-types of}\ \pi\}, 0\leq j_{\tau}\leq s-1\  \\  \text{s.t.}\ \|\tau\|_{spin}=\|
\{\tau-\rho_n^{(j_{\tau})}\}+\rho_c\|\end{array} } 2^{[l_0/2]} E_{\{\tau-\rho_n^{(j_{\tau})}\}},\]
where $l_0=\dim \mathfrak{a}$.
\end{corollary}

\subsection{Structure theory of $\mathrm{GL}(n,\mathbb{H})$}\label{Str_glnh}
We will focus on $G=\mathrm{GL}(n,\mathbb{H})$ from now to the end of Subsection \ref{spin_glnr}. Let $\theta:\mathrm{GL}(n,\mathbb{H})\to \mathrm{GL}(n,\mathbb{H})$ be the Cartan involution given by $$
\theta(g)=(\overline{g}^{-1})^t,
$$
the inverse transpose conjugation of $g\in \mathrm{GL}(n, \mathbb{H})$, then
$K=G^{\theta} = \mathrm{Sp}(n)$.

Choose a maximal compact Cartan subalgebra as follows,

\begin{equation}\label{RealCartan}
\mathfrak{h}_0=\mathfrak{t}_0+\mathfrak{a}_0=\left\{\left[\begin{array}{cccc}a_{1}+\sqrt{-1}c_{1} & & &\\& a_{2}+\sqrt{-1}c_{2} & & \\ & &\ddots & \\& & & a_{n}+\sqrt{-1}c_{n} \end{array} \right]\right\},\end{equation}
where $\mathfrak{t}_0=\{\mathrm{diag}(\sqrt{-1}c_1,\dots,\sqrt{-1}c_n)\ |\ c_j\in \mathbb{R}\}$, $\mathfrak{a}_0=\{\mathrm{diag}(a_1,\dots,a_n)\ |\ a_j\in \mathbb{R}\}.$

For $1\leq k\leq n$, set $e_k$ be the linear functional sending the element in \eqref{RealCartan} to $\sqrt{-1}c_k$, set $f_k$ be the linear functional sending the element in \eqref{RealCartan} to $a_k$. We define a lexicographic ordering by using the set $e_1,\dots,e_n,f_1,\dots,f_n$.

Then the positive roots are
\[\Delta^+=\left\{\begin{array}{ll}e_j+e_k\pm (f_j-f_k), &  j\neq k\\ e_j-e_k\pm (f_j-f_k), & j<k \\ 2e_l, & 1\leq l\leq n\end{array}\right.\]
So the Vogan diagram is as follows:
\begin{center}
\begin{tikzpicture}
 \draw (1.52,2) circle (4pt) node[left]{ $2e_n$}
     (3,3) circle (4pt) node[above]{ $e_{n-1}-e_n+(f_{n-1}-f_n)$}
     (3,1) circle (4pt) node[below]{ $ e_{n-1}-e_n-(f_{n-1}-f_n)$}
     (9,3) circle (4pt) node[above]{ $e_1-e_2+(f_1-f_2)$}
      (9,1) circle (4pt) node[below]{ $e_1-e_2-(f_1-f_2)$}
     (1.52+0.1,2+0.1)--(3-0.1,3-0.1)
     (1.52+0.1,2-0.1)--(3-0.1,1+0.1)
     (3+0.15,3)--(5,3)
     (3+0.15,1)--(5,1)
     (9-0.15,3)--(7,3)
     (9-0.15,1)--(7,1);
 \draw  [dash pattern=on 5pt off 5 pt] (5,3)--(7,3);
 \draw  [dash pattern=on 5pt off 5 pt] (5,1)--(7,1);
 \draw  [<->] (3,3-0.32)--(3,1+0.32);
 \draw  [<->] (9,3-0.32)--(9,1+0.32);

\end{tikzpicture}
\end{center}

We get the positive restricted root system as follows:
\[\Delta^+(\mathfrak{g}, \mathfrak{t})=\Delta^+(\mathfrak{k}, \mathfrak{t}) \cup \Delta^+(\mathfrak{p}, \mathfrak{t}),\]
where \
\[\Delta^+(\mathfrak{p}, \mathfrak{t})=\{e_i\pm e_j \mid 1\leq i<j\leq n\}, \quad
\Delta^+(\mathfrak{k}, \mathfrak{t})=\Delta^+(\mathfrak{p}, \mathfrak{t})\cup \{2e_i\mid 1\leq i\leq n\}.\]
So the Weyl group $W(\mathfrak{g},\mathfrak{t})=W(\mathfrak{k},\mathfrak{t})$. Let us denote the half sum of roots in $\Delta^+(\mathfrak{g}, \mathfrak{t})$ (resp. $\Delta^+(\mathfrak{k}, \mathfrak{t})$, $\Delta^+(\mathfrak{p}, \mathfrak{t})$) by $\rho$ (resp. $\rho_c$, $\rho_n$). Then under the basis $\{e_i, 1\leq i\leq n\}$ of $\mathfrak{t}^*$,
$$
\rho_c=(n, n-1, \dots, 1), \quad \rho_n=\rho-\rho_c=(n-1, n-2, \dots, 0) \in \mathfrak{t}^*.
$$

\subsubsection{$K$-types of $\mathrm{Sp}(n)$} \label{sec-ktypes}
By the Cartan-Weyl theorem, the equivalence class of irreducible representations of $K = \mathrm{Sp}(n)$ (denoted as $\widehat{K}$)
can be parametrized by dominant analytically integral linear functionals on $\mathfrak{t}$. Under the above basis $\{e_i,1\leq i\leq n\}\subset \mathfrak{t}^*$, $\widehat{K}$ can be parametrized by  $\gamma$, where $\gamma=(\gamma_1, \dots, \gamma_n)$ is a decreasing sequence of non-negative integers (cf. (5.28) of \cite{Vogan:1986}). We use $E_{\gamma}$ to denote the $K$-type parametrized by $\gamma$.

\subsubsection{Spin lowest $K$-types in the $\mathrm{GL}(n,\mathbb{H})$ case}
Let $E_{\gamma}$ be a $K$-type, then the spin norm of $E_{\gamma}$ is defined to be
$\| \{\gamma-\rho_n\}+\rho_c\|$. As Corollary \ref{cor-HP}, one has the following.
\begin{corollary} \label{cor-HP-glnH}
Let $G = \mathrm{GL}(n,\mathbb{H})$ and $\pi$ be a unitary $(\mathfrak{g},K)$-module with infinitesimal character $\Lambda \in \mathfrak{h}^*$. Then the following are equivalent to each other.

\begin{itemize}
\item[(i)] $H_D(\pi)$ is non-zero.
\item[(ii)] There exists a $w\in W(\mathfrak{g}, \mathfrak{h})$, and a $K$-type in $\pi$ with highest weight $\tau$ such that
\begin{equation}\label{slkt-cond}
w\Lambda|_{\mathfrak{a}} = 0, \quad \{\tau - \rho_n\} = w\Lambda|_{\mathfrak{t}} - \rho_c.
\end{equation}
\item[(iii)] The highest weight $\tau$ of any spin lowest $K$-type of $\pi$ satisfies  $\|\{\tau-\rho_n\}+\rho_c\|=\| \Lambda\|$.
\end{itemize}

Moreover, if $H_D(\pi)\neq 0$, then
\[H_D(\pi)\cong \mathop{\bigoplus}\limits_{\tau\in \{\text{spin lowest}\ K\text{-types of}\ \pi\} } 2^{[n/2]} E_{\{\tau-\rho_n\}}.\]

\end{corollary}

\subsection{Infinitesimal character}
Suppose $\pi$ is a unitary $(\mathfrak{g},K)$-module satisfying \eqref{eq-HP2} in Corollary \ref{cor-HP},
with an infinitesimal character $\Lambda =  (\lambda_1, \dots, \lambda_{2n}) \in \mathfrak{h}^*$,
so
$$
\Lambda=\sum_{k=1}^n\lambda_k(e_k+f_k)+\sum_{k=1}^{n}\lambda_{n+k}(-e_{n-k+1}+f_{n-k+1}).
$$

Therefore, under the basis $\{e_i|1\leq i\leq n\}$ of $\mathfrak{t}^*$ and the basis $\{f_i| 1\leq i\leq n\}$ of $\mathfrak{a}^*$,
$$
\Lambda|_{\mathfrak{t}} = (\lambda_1 - \lambda_{2n}, \lambda_2 - \lambda_{2n-1}, \dots, \lambda_n - \lambda_{n+1}),$$
$$
\Lambda|_{\mathfrak{a}} = (\lambda_1 + \lambda_{2n}, \lambda_2 + \lambda_{2n-1}, \dots, \lambda_n + \lambda_{n+1}).
$$

Consequently, in the classification of
the Dirac series for $\mathrm{GL}(n, \mathbb{H})$, it suffices to focus on the unitary representations whose infinitesimal
characters $\Lambda \in \mathfrak{h}^*$ are of the form $\Lambda|_{\mathfrak{a}} = 0$ and
\begin{equation} \label{eq-HP}
\Lambda|_{\mathfrak{t}} = (2\lambda_1, \dots, 2\lambda_{n})
\end{equation}
where $2\lambda_1 > \dots > 2\lambda_n $ is a sequence of positive integers. By abuse of notations, we will denote the infinitesimal character $\Lambda \in \mathfrak{h}^*$
by its restriction $\Lambda|_{\mathfrak{t}}$ to $\mathfrak{t}$, from now on.

By Theorem 6.18 and Theorem 17.6 of \cite{Vogan:1986}, all the irreducible unitary representations with half integral infinitesimal characters are unitarily induced by the following three kinds of unitary representations: the special unipotent representations, the $A_{\mathfrak{q}}(\lambda)$ modules, and a kind of Stein complementary series.

\subsection{The special unipotent representations}
The special unipotent representations of $\mathrm{GL}(n,\mathbb{H})$ are of the form
\[\mathrm{Ind}_{\prod_{i=1}^m\mathrm{GL}(n_i,\mathbb{H})}^{\mathrm{GL}(n,\mathbb{H})}(\mathbf{1}\boxtimes\dots\boxtimes \mathbf{1}).\]

\subsection{Cohomological induction}
Let $G$ be a Lie group in the Harish-Chandra class with Cartan involution $\theta$ and the maximal compact subgroup $K$. Denote the Lie algebra of $G$, $K$ by $\mathfrak{g}_0 , \mathfrak{k}_0$. Let $\mathfrak{q}=\mathfrak{l}+\mathfrak{u}$ be a $\theta$-stable parabolic subalgebra of $\mathfrak{g}$. Let $\mathfrak{h}$ be a Cartan subalgebra of $\mathfrak{l}$. Then $\mathfrak{h}$ is also a Cartan subalgebra of $\mathfrak{g}$. Let $L$ be the normalizer of $\mathfrak{q}$ in $G$. Let $\mathfrak{z}$ be the center of $\mathfrak{l}$.

Let us arrange the positive root systems in a compatible way, that is $\Delta(\mathfrak{u},\mathfrak{h})\subseteq \Delta^+(\mathfrak{g},\mathfrak{h})$ and set $\Delta^+(\mathfrak{l},\mathfrak{h})=\Delta(\mathfrak{l},\mathfrak{h})\cap \Delta^+(\mathfrak{g},\mathfrak{h})$. Let $\rho(\mathfrak{u})$ denote the half sum of roots in $\Delta(\mathfrak{u},\mathfrak{h})$. 

Let $\langle \cdot, \cdot \rangle$ denote the nondegenerate symmetric invariant bilinear form on $\mathfrak{g}_0$. This form extends by complexification to all of $\mathfrak{g}$, and by restriction to nondegenerate forms on both $\mathfrak{l}$ and $\mathfrak{h}$. One can use  $\langle \cdot, \cdot \rangle$ to identify $\mathfrak{h}$ and $\mathfrak{h}^*$.
\begin{definition} \label{good range}
With $\langle \cdot, \cdot \rangle$ as above, suppose that the $(\mathfrak{l}, L\cap K)$-module $Z$ has an  infinitesimal character $\Lambda$. We say that $Z$ or $\Lambda$ is in the {\bf weakly good range} or that Z is {\bf weakly good} (relative to $\mathfrak{q}$ and $\mathfrak{g}$) if
\begin{equation}\label{weakly-good}
{\rm Re} \langle \Lambda + \rho(\mathfrak{u}), \alpha \rangle \ge 0, \quad \forall \alpha\in \Delta(\mathfrak{u}, \mathfrak{h}).
\end{equation}
If all the inequalities in \eqref{weakly-good} are strict, then $Z$ is said to be in the {\bf good range}.

We say that $Z$ is {\bf weakly fair} if
\begin{equation}\label{weakly-fair}
{\rm Re} \langle \Lambda + \rho(\mathfrak{u}), \alpha|_{\mathfrak{z}} \rangle \ge 0, \quad \forall \alpha\in \Delta(\mathfrak{u}, \mathfrak{h}),
\end{equation}
If all the inequalities in \eqref{weakly-fair} are strict, then $Z$ is said to be in the {\bf fair range}.

\end{definition}

For a detailed study of cohomological induction, one may refer to \cite{Knapp_Vogan:1995}. Here we use the same notation as \cite{Knapp_Vogan:1995}. The cohomological induction functors $\mathcal{L}_i$ and $\mathcal{R}^i$ lift $(\mathfrak{l}, L \cap K)$ to $(\mathfrak{g},K)$-modules. Let us state the following result (Theorem 0.50 of \cite{Knapp_Vogan:1995}).

\begin{theorem}\label{vanishing}
Let $Z$ be an $(\mathfrak{l},L \cap K)$-module of finite length with  infinitesimal character $\Lambda$, and suppose $Z$ is weakly good. Let $S=\dim \mathfrak{u} \cap \mathfrak{p}$. Then 
\begin{itemize}
\item[(a)] $\mathcal{L}_i(Z)=\mathcal{R}^i(Z)=0$ for $i \neq S$,
\item[(b)] $\mathcal{L}_S(Z) \cong \mathcal{R}^S(Z)$,
\item[(c)] $Z$ irreducible implies $\mathcal{L}_S(Z)$ is irreducible or zero.
\end{itemize}
\end{theorem}

By \cite[Theorem 0.51]{Knapp_Vogan:1995}, we can know that if $Z$ is a weakly good unitary $(\mathfrak{l},L \cap K)$-module, then $\mathcal{L}_S(Z)$ is a unitary $(\mathfrak{g},K)$-module, whose infinitesimal character is $\Lambda+\rho(\mathfrak{u})$. In the special case that $Z$ is a one-dimensional unitary $(\mathfrak{l},L \cap K)$-module, the unitary criterion can be relaxed to weakly fair range. That is, if $Z$ is a weakly fair one dimensional unitary $(\mathfrak{l},L \cap K)$-module, then $\mathcal{L}_S(Z)$ is a unitary $(\mathfrak{g},K)$-module.

\subsection{The $A_{\mathfrak{q}}(\lambda)$ modules} \label{Aqlambda}

We now consider an interesting collection of unitary representations of $\mathrm{GL}(k,\mathbb{H})$.
Let $\lambda=(\frac{b}{2},\dots,\frac{b}{2},-\frac{b}{2},\dots,-\frac{b}{2})\in \mathfrak{h}^*$ for integers $b\geq -k$. Let $\mathfrak{l}=\mathfrak{g}\mathfrak{l}(k,\mathbb{C})$, $L=\mathrm{GL}(k,\mathbb{C})$, and let $\mathbb{C}_{\lambda}$ (or $\mathbb{C}_b$) be the unitary character of $L$:
\[g \mapsto (\frac{\det(g)}{|\det(g)|})^{b},\ \forall g \in L.\]
Let $\mathfrak{q}=\mathfrak{l}+\mathfrak{u}$ be a $\theta$-stable parabolic subalgebra with Levi subalgebra $\mathfrak{l}$. Then
$A_{\mathfrak{q}}(\lambda):=\mathcal{R}^S_\mathfrak{q}(\mathbb{C}_\lambda)$ is an irreducible unitary representation. Since $b\geq -k$, $A_{\mathfrak{q}}(\lambda)$ is weakly fair. Moreover, $A_{\mathfrak{q}}(\lambda)$ has infinitesimal character
\[(b+2k-1,\dots,b+3,b+1)\in\mathfrak{t}^*,\]
and has lowest $K$-type with highest weight $(b+k-1,\dots,b+k-1)$.

\subsection{Unitary dual of $\mathrm{GL}(n,\mathbb{H})$ with half-integral regular infinitesimal characters}
Under the notations in \cite{Vogan:1986}, the Stein complementary series $\sigma(1,t)$ has half integral infinitesimal character
if and only if $t=1/4$.

Therefore, we obtain the following results.
\begin{theorem}\label{half_int_reg}
All the irreducible unitary representations whose infinitesimal characters satisfy condition \eqref{eq-HP} are of the form,
\begin{equation}\label{chain-rep}
\mathrm{Ind}_{\prod_{i=1}^t\mathrm{GL}(k_i,\mathbb{H})\times \mathrm{GL}(r,\mathbb{H})}^{\mathrm{GL}(n,\mathbb{H})} \big(\mathop{\boxtimes}
\limits_{i=1}^t A_{\mathfrak{q}_i}(\lambda_i)\boxtimes \mathbf{1}\big)
\end{equation}
where $A_{\mathfrak{q}_i}(\lambda_i)$ is the representation of $\mathrm{GL}(k_i,\mathbb{H})$ considered in subsection \ref{Aqlambda}.
\end{theorem}
\begin{remark}
Since the infinitesimal character satisfying condition \eqref{eq-HP} is regular, there are no Stein complementary series appearing, and the special unipotent representation showing up must be the trivial representation.
\end{remark}

\section{Dirac series and their spin lowest $K$-types}
In this section, we study Dirac cohomology of all irreducible unitary representations of $\mathrm{GL}(n,\mathbb{H})$.
At first, we recall some notions about chains as given in \cite{Dong_Wong_1:2022}.

\subsection{Describe infinitesimal characters by chains }
We call
$$
\mathcal{C} := (C, C-2, \dots, c+2, c)
$$
a chain which is a descending sequence of non-negative integers with a common difference of two. We use it to denote the infinitesimal character of the unitary representations $A_{\mathfrak{q}}(\lambda)$ or the unipotent representations.

More precisely, we call the chain $\mathcal{C}$ an s-chain or u-chain if $\mathcal{C}$ is the infinitesimal character of an $A_{\mathfrak{q}}(\lambda)$ or some unipotent representation, respectively, and write it as $\mathcal{C}=(C, C-2, \dots, c+2, c)_s$ or $\mathcal{C}=(C, C-2, \dots, c+2, c)_u$, respectively.

\begin{definition}\label{chain}
\begin{enumerate}
\item Two chains $\mathcal{C}_1 = (A, A-2,\dots, a)_{\bullet}$, $\mathcal{C}_2 = (B, B-2, \dots, b)_{\bullet}$ ($\bullet = s \mbox{ or } u$)
are {\bf linked} if
the entries of $\mathcal{C}_1$ and $\mathcal{C}_2$ are disjoint, and either one of the following holds:
\begin{itemize}
\item $A > B > a$; or
\item $B > A > b$; or
\item $\{\mathcal{C}_1, \mathcal{C}_2\} = \{(2k, 2k-2, \dots, 2)_u, (1)_s\}$, where $k$ is a positive integer.
\end{itemize}
\item We say a union of chains  $\displaystyle \bigcup_{i \in I} \mathcal{C}_i$ is {\bf interlaced} if for each $i \neq j$ in $I$,
there exists indices $i = i_0, i_1, \dots, i_k = j$ in $I$ such that $\mathcal{C}_{i_{l-1}}$ and
$\mathcal{C}_{i_{l}}$ are linked for all $1 \leq l \leq k$ (by convention, we also let the single chain $\mathcal{C}$ to be interlaced). Furthermore, we say the interlaced chains {\bf u-interlaced} if it contains exactly one u-chain.
\end{enumerate}
\end{definition}

\begin{remark}
The case $\{\mathcal{C}_1, \mathcal{C}_2\} = \{(2k, 2k-2, \dots, 2)_u, (1)_s\}$ does not appear in the $\mathrm{GL}(n,\mathbb{H})$ case, but happens in the $\mathrm{GL}(n,\mathbb{R})$ ($n$ is odd) case.  
\end{remark}

For convenience, by the abuse of notation, for $a\in \mathbb{R}$, $m\in \mathbb{N}$, let $a^m$ denote the vector $(\underbrace{a,\dots,a}_m)$. For the rest of the content, all of $a^m$ will mean a vector in $\mathbb{R}^m$ rather than a number ``$a$ to the power $m$". 

For $c,d\in \mathbb{R}$ such that $c-d\in 2\mathbb{N}$, let $[c,d]$ denote the chain $(c,c-2,\dots,d)$.
For $c,d\in \mathbb{R}$ such that $c-d\in \mathbb{Z}$, let
\[\{c,d\}:=\begin{cases} (\underbrace{c,c+1,\dots,d}_{d-c+1}), & d>c, \\ (\underbrace{c,c-1,\dots,d}_{c-d+1}), & d\leq c.\end{cases}\]

As in \cite{Dong_Wong_2:2022}, we can associate the union of chains $\bigcup_{i=1}^t \mathcal{A}_i \bigcup \mathcal{U}$ to the induced representation \eqref{chain-rep} where $$\mathcal{A}_i = (b_i + 2k_i - 1, b_i + 2k_i - 3,  \dots ,b_i  + 1)_s;\ \text{and}\ \mathcal{U}=(2r-1,2r-3,\dots,1)_u.$$
\begin{remark}
$\bigcup_{i=1}^t \mathcal{A}_i \bigcup \mathcal{U}$ are precisely the infinitesimal character of the induced representation \eqref{chain-rep} restricted to  $\mathfrak{t}$.
\end{remark}

\subsection{The Dirac series of $\mathrm{GL}(n,\mathbb{H}$) and their spin lowest $K$-types}

\begin{theorem}\label{thm-main}
Let $\pi$ be an irreducible unitary $(\mathfrak{g},K)$-module of $\mathrm{GL}(n,\mathbb{H})$. Then $\pi$ has nonzero Dirac cohomology if and only if $\pi$ is of the form $$\mathrm{Ind}_{\prod_{i=1}^t\mathrm{GL}(k_i,\mathbb{H})\times \mathrm{GL}(r,\mathbb{H})}^{\mathrm{GL}(n,\mathbb{H})} \big(\mathop{\boxtimes}\limits_{i=1}^t A_{\mathfrak{q}_i}(\lambda_i)\boxtimes \mathbf{1}\big),$$ 
where $A_{\mathfrak{q}_i}(\lambda_i)$ is the representation of $\mathrm{GL}(k_i,\mathbb{H})$ considered in Subsection \ref{Aqlambda}, such that the coordinates of the chain $\Lambda = \bigcup_{i=1}^t \mathcal{A}_i \cup \mathcal{U}$ associated to $\pi$ are different from each other.

Moreover, if $\pi$ has nonzero Dirac cohomology, then $\pi$ contains unique spin lowest $K$-type with multiplicity one, which can be obtained by the following way:
\begin{itemize}
\item[(i)] Let $\mathcal{L}_i=\underbrace{(b_i+k_i-1,\dots,b_i+k_i-1)}_{k_i}$, so $\bigcup_{i=1}^t \mathcal{L}_i$ is the extremal weight of the lowest $K$-type of $$\mathrm{Ind}_{\prod_{i=1}^t \mathrm{GL}(k_i,\mathbb{H})}^{\mathrm{GL}(k,\mathbb{H})}\big(\mathop{\boxtimes}\limits_{i=1}^t A_{\mathfrak{q}_i}(\lambda_i)\big),\ k=\sum_{i=1}^t{k_i}.$$ 
Then the highest weight of its spin lowest $K$-type $\bigcup_{i=1}^t \mathcal{T}_i$ can be obtained from $\bigcup_{i=1}^t \mathcal{L}_i$ by applying Algorithm 
\ref{algorithm1} (i.e. Algorithm 2.2 of \cite{Dong_Wong:2021}).
\item[(ii)] Let $\mathcal{T}_u=(\underbrace{0,0,\dots,0}_r)$ be the highest weight of the trivial $K$-type
of the trivial representation of $\mathrm{GL}(r,\mathbb{H})$. \item[(iii)]  One can obtain $\mathcal{S}_i$, $\mathcal{S}_u$  from $\mathcal{T}_i$ ,$\mathcal{T}_u$ by applying Algorithm 
\ref{algorithm2} (i.e. the method of Theorem 3.1 of \cite{Dong_Wong_1:2022}). Then
\[\mathcal{S}=(\mathcal{S}_1;\dots;\mathcal{S}_t;\mathcal{S}_u)\] is the highest weight of the unique spin lowest K-type of $\pi$.
\end{itemize}
\end{theorem}

\begin{proof}
Firstly, by Theorem \ref{half_int_reg}, $\pi$ has nonzero Dirac cohomology only if $\pi$ is of the form in the statement. We show that such $\pi$ has non-zero Dirac cohomology, and contains a spin lowest $K$-type which can be obtained by the way in the statement: 

Actually, by a similar argument to Theorem 3.1 of \cite{Dong_Wong_1:2022}, or alternatively by the proof of Theorem \ref{uni_m1_GLnH}, the $\mathcal{S}$ defined in the statement is the highest weight of a $K$-type of $\pi$ and the $\mathcal{S}$ satisfies \eqref{slkt-cond}. Hence, $\pi$ has non-zero Dirac cohomology as (iii) of Corollary \ref{cor-HP-glnH}.

As for the uniqueness and multiplicity freeness of spin lowest $K$-type, it will be proved in Theorem \ref{uni_m1_GLnH}.
\end{proof}

We are now to describe the algorithm for the spin-lowest $K$-type of the module $\pi$ in Theorem \ref{thm-main}.

\begin{algorithm}\label{algorithm1}
Let $\mathcal{C}_1=(a,a-2,\dots,a-2p+2)$ and $\mathcal{C}_2=(b,b-2,\dots,b-2q+2)$ be a pair of linked chains such that their coordinates are different from each other. Let $\zeta_1=(a-p)^p$ and $\zeta_2=(b-q)^q$ corresponding to $\mathcal{C}_1$ and $\mathcal{C}_2$.
We will change the coordinates of $\zeta_1$ and $\zeta_2$ by the following rule: 

Case (I):  $a > b \geq b-2q+2 > a-2p+2$. Assume that
$a-2m+2 > b > a-2m$ for some $1\leq m\leq p-1$, then change the coordinates of $\zeta_1$ and $\zeta_2$ into
\[\begin{array}{c}\zeta_1'=\zeta_1+\big(0^{m};p-m,p-m-1,\dots,p-m-q+1;0^{p-m-q}\big), \vspace{0.8em} \\
\zeta_2'=\zeta_2+\big(-(p-m),-(p-m-1),\dots,-(p-m-q+1)\big).\end{array}\]
  
Case (II): $a > b > a-2p+2 > b-2q+2$. Assume that
$a-2m > b > a-2m-2$ for some $0\leq m\leq p-2$, then change the coordinates of $\zeta_1$ and $\zeta_2$ into
\[\begin{array}{c}\zeta_1'=\zeta_1+\big(0^{m};1,2,\dots,(p-m)\big), \vspace{0.8em} \\ 
\zeta_2'=\zeta_2+\big(-1,-2,\dots,-(p-m);0^{q-p+m}\big).\end{array}\] 
\end{algorithm}

 \begin{algorithm}\label{algorithm2}
Let $\mathcal{C}=(a,a-2,\dots,a-2p+2)$ with $a-2p+2>1$, and $\mathcal{U}=(2r-1,2r-3,\dots,1)$ be a pair of linked chains such that their coordinates are different from each other.
We will change the coordinates of $\zeta=(a-p)^p$ and $\zeta_0=0^r$ by the following rule: 

Case (I):  $a-2m+2>2r-1>a-2m$ for some $0\leq m\leq p-2$. 
Then we change the coordinates of $\zeta$ and $\zeta_0$ into
\[\begin{array}{c}\zeta'=\zeta+\big(0^{m},p-m,p-m-1, \dots,1\big), \vspace{0.8em}\\
\zeta_0'=\zeta_0+(p-m,p-m-1,\dots,1;0^{r-p+m}).\end{array}\]
 
Case (II): $2r-2x+1>a-2p+2>2r-2x-1$ for some $1\leq x\leq r-1$
Then we change the coordinates of $\zeta$ and $\zeta_0$ into
\[\begin{array}{c}\zeta'=\zeta+\big(x,x-1,\dots,x-p+1\big), \vspace{0.8em} \\
\zeta_0'=\zeta_0+(0^{x-p},x,x-1,\dots,x-p+1,0^{r-x}).\end{array}\]
\end{algorithm}

\subsection{FS-scattered representations}

We conclude this section by classifying the FS-scattered representations of $\mathrm{GL}(n,\mathbb{H})$.

Recall that the FS-scattered representations are the representations of $\widehat{G}^d$ which cannot be cohomologically induced from any representations (not necessarily from $\widehat{G}^d$) of any proper $\theta$-stable Levi subgroup in the weakly good range. See \cite{Ding_Dong:2020} for details.
The following theorem tells that the Dirac cohomology behaves well when doing cohomological induction under the weakly good range.

\begin{theorem}[\cite{Dong_Huang:2015}, Theorem 5.7]
Let $G$ be a real reductive Lie group in Harish-Chandra class. Suppose that the irreducible $(\mathfrak{l},L\cap K)$-module $Z$ has infinitesimal character $\lambda\in i\mathfrak{t}_0^*$ which is weakly good. Then we have the following $\widetilde{K}$-module isomorphism
\[\mathcal{L}^{\widetilde{K}}_S(H_D(Z)\otimes \mathbb{C}_{-\rho(\mathfrak{u}\cap \mathfrak{p})})\cong H_D(\mathcal{L}_S(Z)).\]
Moreover, the spin lowest $K$-types of $\mathcal{L}_S(Z)$ are contained in $\mathcal{L}_S^K(E_Z)$ for some spin lowest $K$-type $E_Z$ of $Z$. If $Z$ has unique spin lowest $K$-type with multiplicity one, then either $\mathcal{L}_S(Z)$ has zero Dirac cohomology or it has unique spin lowest $K$-type with multiplicity one.
\end{theorem}

Therefore, we can focus on the FS-scattered representations to understand the whole Dirac series. In the last section, in order to prove the multiplicity-free of the spin lowest $K$-type of any Dirac series of $\mathrm{GL}(n,\mathbb{H})$ or $\mathrm{GL}(n,\mathbb{R})$, we only need to prove it for the FS-scattered representations. The following propositions can be argued similarly to Section 5 of  \cite{Dong_Wong_2:2022}

\begin{proposition}\label{FS-sc}
Let $\pi$ be an irreducible unitary $(\mathfrak{g},K)$-module of $\mathrm{GL}(n,\mathbb{H})$ with nonzero Dirac cohomology. Then
$\pi$ is an FS-scattered representation if and only if its corresponding chains are $u$-interlaced.
\end{proposition}
\begin{proof}
For necessity: Assume that the chains of $\pi$ do  not contain the $u$-chain, then $\pi$ can be written as a cohomological induction from a module $Z$ (see the beginning of Subsection \ref{spin_glnr}) of $\mathrm{GL}(n,\mathbb{C})$ in weakly good range.  Suppose the chains  $\bigcup_{i=1}^t \mathcal{A}_i \cup \mathcal{U}$ of $\pi$ are not interlaced. We may decompose them as 
\[\bigcup_{i=1}^t \mathcal{A}_i \cup \mathcal{U}=\mathcal{C}_1\cup (\mathcal{C}_2\cup\mathcal{U}),\]
such that for all coordinates $x\in\mathcal{C}_1$, $y\in \mathcal{C}_2\cup\mathcal{U}$, one has $x\geq y$. Then we can write $\pi$ as a cohomological induction from a representation $\pi_1\boxtimes\pi_2$, where $\pi_1$ is associated to the chain $\mathcal{C}'_1$ which is a shift of $\mathcal{C}_1$ by some constant, $\pi_2$ is associated to the chains $\mathcal{C}'_2\cup\mathcal{U}$, where $\mathcal{C}'_2$ is a shift of $\mathcal{C}_2$ by some constant. By our assumptions about $\mathcal{C}_1$ and $\mathcal{C}_2$, the infinitesimal character $\mathcal{C}'_1\cup\mathcal{C}'_2\cup\mathcal{U}$ of $\pi_1\boxtimes\pi_2$ satisfies \eqref{weakly-good}. Therefore, $\pi$ is cohomologically induced in the weakly good range.

For sufficiency: Assume $\pi$ is cohomologically induced from an irreducible $(\mathfrak{l},L\cap K)$-module $\pi_L$ in the weakly good range. By cohomological induction in stages, one can assume  $L=\mathrm{GL}(n_1,\mathbb{C})\times \mathrm{GL}(n_2,\mathbb{H})$ with $n_1+n_2=n$. Notice that $\pi_L$ is also unitary since it is in the weakly good range. By the classification of unitary representations of $L$, we know that $\pi_L$ is also corresponds to some chains, from which one can also get the chains of $\pi$. Since the chains of $\pi$ are unique (by classification of unitary dual of $\mathrm{GL}(n,\mathbb{H})$), one can get the chains of $\pi$ are not $u$-interlaced from the condition that $\pi_L$ is in the weakly good range. 
\end{proof}

\begin{proposition}
There are totally $2^{n-2}$ FS-scattered representations of $\mathrm{GL}(n,\mathbb{H})$ for $n\geq 2$.
\end{proposition}
\begin{proof}
Notice that the number of the $u$-interlaced chains are one-one corresponding to the interlaced chains whose minimal coordinates are $1$. Now the proposition follows from \cite[Corollary 3.9]{Dong_Wong:2021}.
\end{proof}

\begin{example}

We now list all FS-scattered representations of $\mathrm{GL}(4,\mathbb{H})$. All u-interlaced chains for $\mathrm{GL}(4,\mathbb{H})$ are as follows:
\[\ \begin{aligned}&(7,5,3,1)_u,\end{aligned}\qquad\begin{aligned}&\qquad(4)_s   \\  & \quad (5,\quad3,\quad1)_u      \end{aligned},\qquad\begin{aligned}&\qquad\qquad(2)_s   \\  & \quad (5,\quad3,\quad1)_u     \end{aligned},\qquad\begin{aligned} \quad(4,\quad2)_s \quad\\ \quad(3,\quad1)_u
\end{aligned}.\]
We can list the Dirac series in the form of \eqref{chain-rep} corresponding to the above four chains respectively, and use algorithms \ref{algorithm1} and \ref{algorithm2} to get the spin lowest $K$-types:
\begin{table}[H]
\caption{FS-scattered representations of $\mathrm{GL}(4,\mathbb{H})$}
\begin{tabular}{lcc}
\hline
$\text{Parameters}$&\qquad\qquad \qquad $\text{Lowest $K$-types}$&\qquad \qquad \qquad $\text{Spin lowest $K$-types}$\\ \hline
$(7,5,3,1)_u$ &\qquad \qquad \qquad $(0,0,0,0)$&\qquad \qquad \qquad $(0,0,0,0)$\\
$(4)_s\cup(5,3,1)_u$&\qquad \qquad \qquad $(3,0,0,0)$&\qquad \qquad \qquad $(4,1,0,0)$\\
$(2)_s\cup(5,3,1)_u$&\qquad\qquad \qquad $(1,0,0,0)$&\qquad \qquad\qquad $(3,2,0,0)$\\
$(4,2)_s\cup(3,1)_u$&\qquad \qquad\qquad $(2,2,0,0)$&\qquad \qquad\qquad $(3,2,1,0)$\\
\hline

\end{tabular}
\end{table}

\end{example}

\section{The uniqueness and multiplicity-one of spin lowest $K$-type}

\subsection{The uniqueness and multiplicity-one of spin lowest $K$-type in the $\mathrm{GL}(n,\mathbb{H})$ case}\label{spin_glnr}

 Recall that under the coordinates in subsection \ref{Str_glnh}, we say that a weight of $\mathfrak{t}$, $(a_1,\dots,a_n)\in \mathbb{Z}^n$, is dominant of type $A$ (resp. type $C$) if $a_1\geq \dots\geq a_n$ (resp. $a_1\geq \dots\geq a_n\geq 0$). The Weyl group $W(\mathfrak{k},\mathfrak{t})$ can be identified with $S_n\ltimes (\mathbb{Z}/2\mathbb{Z})^{n}$ with $S_n$ permuting the coordinates and $(\mathbb{Z}/2\mathbb{Z})^{n}$ flipping signs. 
 
 Continuing with the notation of Theorem \ref{thm-main}. We write $\pi$ as a cohomologically induced module: Let
 \[L_1=\mathrm{GL}(k,\mathbb{C}), \ L_2=\mathrm{GL}(r,\mathbb{H}),\ \text{and}\ L=L_1\times L_2\]
 and let $\mathfrak{l}$ be the complexification of the Lie algebra of $L$. 
 Let $\mathfrak{q}=\mathfrak{l}\oplus\mathfrak{u}$  be the $\theta$-stable parabolic subalgebra. Let
\[W^1=\{w\in W(\mathfrak{k},\mathfrak{t})\ |\ \langle w\rho_c,\alpha\rangle>0, \forall \alpha\in \Delta^+(\mathfrak{l}\cap \mathfrak{k})\},\]
which consists of the elements $w\in S_n\ltimes (\mathbb{Z}/2\mathbb{Z})^{n}$ such that $w(n,\dots,1)=(x_1,\dots,x_k\ |\ y_1,\dots,y_r)$ with $x_1>\dots>x_k$ and $y_1>\dots>y_r>0$.

Let $Z$ be the unitary $L_1$-representation
$\mathrm{Ind}_{\prod_{i=1}^t \mathrm{GL}(k_i,\mathbb{C})}^{\mathrm{GL}(k,\mathbb{C})}\big(\mathop{\boxtimes}\limits_{i=1}^t \mathbb{C}_{\beta_i'}\big)$, where $\mathcal{C}_{\beta_i'}$ is the unitary character $(\frac{\det}{|\det|})^{b_i+k_i-k-2r}$ of $\mathrm{GL}(k_i,\mathbb{C})$.
By Theorem 17.6 of \cite{Vogan:1986}, one can express
$
\pi \cong \mathcal{R}_{\mathfrak{q}}^N(Z \boxtimes \mathbf{1})
$
as a cohomologically induced module in the weakly fair range, where $N=\dim \mathfrak{u} \cap \mathfrak{k}.$

Then by Theorem 6.3.12 of \cite{Vogan:1981}, the Blattner-type formula gives the multiplicity of any $K$-type $E_{\tau}$ in $\pi$:
$$
\left[\pi: E_{\tau}\right]_K = \sum_{q} (-1)^q \left[\left( Z \boxtimes \mathbf{1} \right) \otimes \mathbb{C}_{2\rho(\mathfrak{u} \cap \mathfrak{p})} \otimes S(\mathfrak{u} \cap \mathfrak{p}): H^q(\mathfrak{u} \cap \mathfrak{k}, E_{\tau})\right]_{L \cap K},
$$
where
\begin{equation}\label{coh_lk}
H^q(\mathfrak{u}\cap \mathfrak{k},E_{\tau})=\mathop{\bigoplus}\limits_{l(w_1)=q,w\in W^1}E^{L\cap K}_{w(\tau+\rho_c)-\rho_c}\end{equation}
where
$S(\mathfrak{u} \cap \mathfrak{p})$ is the symmetric algebra of $\mathfrak{u} \cap \mathfrak{p}$, and $\mathbb{C}_{2\rho(\mathfrak{u} \cap \mathfrak{p})}\cong {\det}^{k+2r-1} \boxtimes \mathbf{1}$ is the character of $L\cap K$ acting on the top wedge product of $\mathfrak{u} \cap \mathfrak{p}$. 

Let
\[Z^{\#}=Z\otimes_{\mathbb{C}}(\frac{\det}{|\det|})^{k+2r-1},\] 
then $(Z\boxtimes \mathbf{1})\otimes \mathbb{C}_{2\rho(\mathfrak{u} \cap \mathfrak{p})}\cong Z^{\#}\boxtimes \mathbf{1}$ as $L\cap K$-modules.

The following well-known lemma will be used frequently in the remaining. Its proof follows from highest weight theory. 
\begin{lemma}\label{high_lem}
Let $V_1$ and $V_2$ be two finite-dimensional irreducible highest weight module of a reductive Lie algebra, then the highest weights in $V_1\otimes V_2$ is of the form $\mu_1+\mu_2$, where $\mu_1$ is highest weight of $V_1$, and $\mu_2$ is a weight of $V_2$.
\end{lemma}

Now, let us describe the structure of $L\cap K$-module $S(\mathfrak{u}\cap\mathfrak{p})$ in the following two lemmas, which will used to prove Theorem \ref{uni_m1_GLnH}.

\begin{lemma}\label{sup_glnh}
Let
\[\begin{aligned}\mathcal{E}_1 & =\mathop{\bigoplus}\limits_{\tiny\begin{array}{c} a_1 \geq \dots \geq a_x\geq 0,\\ x\leq \min\{k,r\}\end{array}}
\big(E_{\tiny(\underbrace{a_1,\dots,a_x,0,\dots,0}_{k})}^{\mathrm{U}(k)}\boxtimes E_{\tiny(\underbrace{a_1,\dots,a_x,0,\dots,0}_{r})}^{\mathrm{Sp}(r)}\big),\\ \mathcal{E}_2& =\mathop{\bigoplus}\limits_{\tiny\begin{array}{c}a_1\geq \dots \geq a_k\geq 0; c_1\geq \dots \geq c_{r}\geq 0;\\ \sum_{i=1}^ka_i>\sum_{i=1}^{r}c_i\end{array}}   \big(E_{(a_1,\dots,a_k)}^{\mathrm{U}(k)}\boxtimes E_{(c_1,\dots,c_{r})}^{\mathrm{Sp}(r)}\big) \end{aligned}\]
 be the $\mathrm{U}(k)\times \mathrm{Sp}(r)$-modules. Then each $L\cap K\cong \mathrm{U}(k)\times \mathrm{Sp}(r)$-type in $\mathcal{E}_1$  shows up in $S(\mathfrak{u}\cap\mathfrak{p})$ with multiplicity one. 
Moreover, the set of $\mathrm{U}(k)\times \mathrm{Sp}(r)$-types in $\mathcal{E}_1\oplus \mathcal{E}_2$ contain the set of $L\cap K\cong \mathrm{U}(k)\times \mathrm{Sp}(r)$-types showing up in $S(\mathfrak{u}\cap\mathfrak{p})$, where the sets are not counting the multiplicities.
\end{lemma}

\begin{proof}
Firstly, as $L\cap K$ representation,
\[\mathfrak{u}\cap\mathfrak{p} \cong \big(E_{(1,0,\dots,0)}^{L_1\cap K}\boxtimes E_{(1,0,\dots,0)}^{L_2\cap K}\big)\oplus
\big(E_{(1,1,0,\dots,0)}^{L_1\cap K}\boxtimes E_{(0,\dots,0)}^{L_2\cap K}\big).\]

As $\mathrm{GL}(k,\mathbb{C})\times \mathrm{GL}(2r,\mathbb{C})$-module, by Corollary 5.6.6 of \cite{Goodman_Wallach:2009}, we have
\[S\big(E_{(1,0,\dots,0)}^{\mathrm{U}(k)}\boxtimes E^{\mathrm{U}(2r)}_{(1,0,\dots,0)}\big)\cong \bigoplus_{a_1\geq \dots \geq a_x\geq 0}\big(E_{(a_1,\dots,a_x,0,\dots,0)}^{\mathrm{U}(k)}\boxtimes E_{(a_1,\dots,a_x,0,\dots,0)}^{\mathrm{U}(2r)}\big).\]

By the Lemma \ref{GLtoSp} below, any $L\cap K\cong \mathrm{U}(k)\times \mathrm{Sp}(r)$-type in $\mathcal{E}_1$  shows up with multiplicity one in 
\[S\big(E^{\mathrm{U}(k)}_{(1,0,\dots,0)}\boxtimes E^{\mathrm{U}(2r)}_{(1,0,\dots,0)}\big)|_{\mathrm{U}(k)\times \mathrm{Sp}(r)}\cong S\big(E_{(1,0,\dots,0)}^{L_1\cap K}\boxtimes E_{(1,0,\dots,0)}^{L_2\cap K}\big),\]
of which the $\mathrm{U}(k)\times \mathrm{Sp}(r)\cong L\cap K$-types is contained in the set of $\mathrm{U}(k)\times \mathrm{Sp}(r)\cong L\cap K$-type in $\mathcal{E}_1\oplus \mathcal{E}_2$, without counting the multiplicities. Moreover, when $l>0$, the $L\cap K\cong \mathrm{U}(k)\times \mathrm{Sp}(r)$-types showing up in
\[S^m(E^{L_1\cap K}_{(1,0,\dots,0)}\boxtimes E^{L_2\cap K}_{(1,0,\dots,0)})\otimes S^l(E^{L_1\cap K}_{(1,1,\dots,0)}\boxtimes E^{L_2\cap K}_{(0,0,\dots,0)})\]
are contained in $\mathcal{E}_2$.  Since $S(\mathfrak{u}\cap\mathfrak{p}) \cong \mathop{\oplus}\limits_{m,l\in \mathbb{N}}\Big(S^m\big(E_{(1,0,\dots,0)}^{L_1\cap K}\boxtimes E_{(1,0,\dots,0)}^{L_2\cap K}\big)\otimes
S^l\big(E_{(1,1,0,\dots,0)}^{L_1\cap K}\boxtimes E_{(0,\dots,0)}^{L_2\cap K}\big)\Big)$, one get the lemma immediately.
\end{proof}

\begin{lemma}\label{GLtoSp}
Any irreducible module of the restriction of $\mathrm{GL}(2r,\mathbb{C})$-module $E^{\mathrm{GL}(2r,\mathbb{C})}_{(a_1,\dots,a_{2r})}$, $a_1\geq \dots\geq a_{2r}$, to $\mathrm{Sp}(2r,\mathbb{C})$ has highest weight
\[(a_1-a_{2r},\dots,a_{r}-a_{r+1})+(\mu_1,\dots,\mu_r),\] where $(\mu_1,\dots,\mu_r)$ satisfies $\sum_{i=1}^j\mu_i\leq 0$ for any $1\leq j\leq r$.
Moreover,  if $a_{r+1}=\dots=a_{2r}=0$, then
\[E^{\mathrm{GL}(2r,\mathbb{C})}_{(a_1,\dots,a_{r},0,\dots,0)}|_{\mathrm{Sp}(2r,\mathbb{C})}=E^{\mathrm{Sp}(2r,\mathbb{C})}_{(a_1,\dots,a_r)}\oplus \bigoplus_{\sum_{i=1}^r c_i<\sum_{i=1}^r a_i} m_{(c_1,\dots,c_r)}E^{\mathrm{Sp}(2r,\mathbb{C})}_{(c_1,\dots,c_r)},\]
for some multiplicities $m_{(c_1,\dots,c_r)}\in\mathbb{N}$. \end{lemma}
\begin{proof}
Assume that $v$ is the highest weight vector of an irreducible $\mathrm{Sp}(2r,\mathbb{C})$-module which appears in $E^{\mathrm{GL}(2r,\mathbb{C})}_{(a_1,\dots,a_{2r})}$, and $v$ has weight $(c_1,\dots,c_r)$ with $c_1\geq \dots \geq c_r\geq 0$. Decompose $v$ into the sum of the weight vectors in $\mathrm{GL}(2r,\mathbb{C})$-module as
\[v=v_1+\dots+v_m\]
where $v_i$ and $v_j$ have different weights for any $1\leq i\neq j\leq m$.
For any $v_i$, let $(d_1,\dots,d_{2r})$ be the weight of $v_i$, then
$c_j=d_j-d_{2r-j}$, for all $1\leq j\leq r$.

Let $\mathfrak{b}$ be a Borel subalgebra of $\mathfrak{g}\mathfrak{l}(2r,\mathbb{C})$ which contains a Borel subalgebra $\mathfrak{b}_0$ of $\mathfrak{sp}(2r,\mathbb{C})$. Then the highest weight vector of the $\mathrm{GL}(2r,\mathbb{C})$-module is in $U(\mathfrak{b})v$, where $U(\mathfrak{b})$ is the universal inveloping algebra of $\mathfrak{b}$.
Under the notation in \cite[Section II.1]{Knapp:2002},
\[\mathfrak{b}=\mathfrak{b}_0 \oplus \bigoplus_{1\leq i<j\leq r} E_{e_i-e_{r+j}} \oplus \bigoplus_{1\leq i<j\leq r} E_{e_i-e_j},\]
and $v$ is highest vector of $\mathfrak{sp}(2r,\mathbb{C})$-module, one has 
\[U(\mathfrak{b})v=U(\mathop{\bigoplus}\limits_{1\leq i<j\leq r} E_{e_i-e_{r+j}} \oplus \mathop{\bigoplus}\limits_{1\leq i<j\leq r} E_{e_i-e_j})v.\] 
In particular, there exists some $1\leq i_0\leq m$ such that the highest weight vector of $\mathfrak{gl}(2r,\mathbb{C})$-module is in $U(\mathop{\bigoplus}\limits_{1\leq i<j\leq r} E_{e_i-e_{r+j}} \oplus \mathop{\bigoplus}\limits_{1\leq i<j\leq r} E_{e_i-e_j})v_{i_0}$. Hence, 
\[(a_1,\dots,a_{2r})=(d_1,\dots,d_{2r})+\sum_{1\leq i<j\leq r}f_{ij}(e_i-e_{r+j}) +\sum_{1\leq i<j\leq r}g_{ij}(e_i-e_j),\]
for some $f_{ij},g_{ij}\geq 0$.

Therefore, $(c_1,\dots,c_r)=(d_1-d_{2r},\dots,d_{r}-d_{r+1})=(a_1-a_{2r},\dots,a_{r}-a_{r+1})+(\mu_1,\dots,\mu_r)$, where $(\mu_1,\dots,\mu_r)$ satisfies $\sum_{i=1}^j\mu_i\leq 0$ for any $1\leq j\leq r$.

Moreover,  if $\sum c_i=\sum a_i$, then $a_{r+1}=\dots=a_{2r}=0$. By the Littlewood-Richardson branching rules, see Theorem 9.76 of \cite{Knapp:2002},  we know
\[E^{\mathrm{GL}(2r,\mathbb{C})}_{(a_1,\dots,a_{r},0,\dots,0)}|_{\mathrm{Sp}(2r,\mathbb{C})}=E^{\mathrm{Sp}(2r,\mathbb{C})}_{(a_1,\dots,a_r)}\oplus \bigoplus_{\sum_{i=1}^r c_i<\sum_{i=1}^r a_i} m_{(c_1,\dots,c_r)}E^{\mathrm{Sp}(2r,\mathbb{C})}_{(c_1,\dots,c_r)}.\]
Now the lemma follows immediately.
\end{proof}

The following lemma is about the $K$-types of the induced module of $\mathrm{GL}(k,\mathbb{C})$, which will be used to prove Theorem \ref{uni_m1_GLnH} and Theorem \ref{uni_m1_GLnR}.

\begin{lemma}\label{kt_of_zs}
Let $V$ be the representation of $\mathrm{GL}(k,\mathbb{C})$
\[ \mathrm{Ind}_{\prod_{i=1}^t\mathrm{GL}(k_i,\mathbb{C})}^{\mathrm{GL}(k,\mathbb{C})}\big(
\mathop{\boxtimes}\limits_{i=1}^t\mathbb{C}_{\beta_i}\big),\]
where $\mathbb{C}_{\beta_i}$ is the unitary character $(\frac{\det}{|\det|})^{\beta_i}$, $\beta_i\in \mathbb{N}$, of $\mathrm{GL}(k_i,\mathbb{C})$. Let $\zeta_V$ denote $(\beta_1^{k_1};\dots;\beta_t^{k_t})$. Let $\gamma$ be a highest weight of a $\mathrm{U}(k)$-type $E_{\gamma}$ of $V$, then
\begin{itemize}
\item[(i)] $\gamma$ is in the set
\[\big\{\zeta_V+d\ |\
 d\in \mathbb{N}\langle\Delta(\mathfrak{u}_0)\rangle\big\},\]
 where $\mathfrak{u}_0$ is the nilpotent radical of the standard parabolic subalgebra whose Levi factor is $\mathfrak{l}_0=\bigoplus\limits_{i=1}^t\mathfrak{g}\mathfrak{l}(k_i,\mathbb{C})$, and $\mathbb{N}\langle\Delta(\mathfrak{u}_0)\rangle$ are all linear combinations of the roots in $\mathfrak{u}_0$ with coefficients in $\mathbb{N}$.

\item[(ii)] 
Let $d$ be the component of $\gamma$ in (i), so $d=(d_1;\dots;d_t)$, $d_i \in \mathbb{Z}^{k_i}$.
Write $d_1=(\delta_1,\dots,\delta_{k_1})$, $d_2=(\nu_1,\dots,\nu_{k_2})$. Assume that $d_2$ satisfies that there exist some $1\leq s \leq k_2$, s.t.
\begin{equation}\label{cond_dq+1}
\nu_{k_2-x+1}<0,\ \forall\   1 \leq x\leq s,\end{equation}
that is, the last $s$ coordinates of $d_2$ are $<0$.
Then $k_1\geq s$ and 
\begin{equation}\label{htilde}
\delta_x \geq -\nu_{k_2-x+1}, \forall\ 1\leq x\leq s.
\end{equation}

\item[(iii)] Following (ii), assume that 
\[d_1=(\delta_1,\dots,\delta_s,0^{k_1-s})\  \text{and}\ d_2=(\nu_1,\dots,\nu_{k_2-s},-\delta_s,\dots,-\delta_1)\] 
with $\delta_{\bullet}\geq 0$. Let $\varepsilon\in \mathbb{N}$ such that $\varepsilon\leq \delta_1$. Define $\Psi: \mathbb{R}^k\to \mathbb{R}^{k-k_1-s}$ by 
\[\Psi\big((y_1,\dots,y_k)\big)=(y_{k_1+s+1},\dots,y_k)-(\varepsilon^{k_2-s};0^{k-k_1-k_2}).\]
By assumption on $\varepsilon$,  $\Psi(\gamma)$ is dominant of type $A$.

Let $V'=\mathrm{Ind}_{\mathrm{GL}(k_2-s,\mathbb{C})\times \prod_{i=3}^t\mathrm{GL}(k_i,\mathbb{C})}^{\mathrm{GL}(k-k_1-s,\mathbb{C})}\big(
\mathbb{C}_{\beta_2-\varepsilon}\boxtimes(\mathop{\boxtimes}\limits_{i=3}^t\mathbb{C}_{\beta_i})\big)$ be the induced representation of $\mathrm{GL}(k-k_1-s,\mathbb{C})$.
Then we have 
\[\big[E_{\gamma}^{\mathrm{U}(k)} : V\big]_{\mathrm{U}(k)}=\big[E_{\Psi(\gamma)}^{\mathrm{U}(k-k_1-s)} : V'\big]_{\mathrm{U}(k-k_1-s)}.\]

\end{itemize}
\end{lemma}
\begin{proof}
For (i): By Frobenius reciprocity, $E_{\gamma}|_{\mathfrak{l}_0}$ contains submodule $\big(
\mathop{\boxtimes}\limits_{i=1}^t\mathbb{C}_{\beta_i}\big)$. Let $\mathfrak{b}_0$ be the standard Borel subalgebra.

The highest weight vector of $E_{\gamma}$ is contained in
\[U(\mathfrak{b}_0)\big(
\mathop{\boxtimes}\limits_{i=1}^t\mathbb{C}_{\beta_i}\big)=U(\mathfrak{u}_0)\big(\mathop{\boxtimes}
\limits_{i=1}^t\mathbb{C}_{\beta_i}\big),\]  where $U(\bullet)$ denotes the enveloping algebra of $\bullet$. So we get (i).

For (ii): Given any $\eta_1\in \mathbb{Z}^{m_1}$ and $\eta_2\in\mathbb{Z}^{m_2}$, $m_1, m_2\in \mathbb{N}_+$, define 
\[\mathcal{I}(\eta_1,\eta_2):=\mathrm{Ind}_{\mathrm{U}(m_1)\times \mathrm{U}(m_2)}^{\mathrm{U}(m_1+m_2)} \big(E_{\eta_1}^{\mathrm{U}(m_1)}\boxtimes E_{\eta_2}^{\mathrm{U}(m_2)}\big).\]

Let $\gamma_1=(\beta_1^{k_1})\in \mathbb{Z}^{k_1}, \gamma_2\in \mathbb{Z}^{k_1+k_2},\dots,\gamma_{t-1}\in \mathbb{Z}^{k_1+\dots+k_{t-1}}, \gamma_t=\gamma$ be the dominant weights of type $A$ such that $E_{\gamma_{j}}^{\mathrm{U}(\sum_{i=1}^jk_i)}$ shows up in $\mathcal{I}_j:=\mathcal{I}(\gamma_{j-1},\beta_j^{k_j})$ for all $2\leq j\leq t$. Similar to the proof of (i), by definition of $\gamma_j$, one has
\begin{equation}\label{eta_y+1}
\gamma_{j}\in (\gamma_{j-1};\beta_{j}^{k_{j}})+\mathbb{N}\langle\Delta(\mathfrak{u}_{0,j})\rangle,\ \forall 2\leq j\leq t,
\end{equation}
where $\mathfrak{u}_{0,j}$ is the nilpotent radical of the standard parabolic subalgebra of $\mathfrak{g}\mathfrak{l}(\sum\limits_{i=1}^jk_i,\mathbb{C})$, whose Levi factor is 
$\mathfrak{g}\mathfrak{l}(\sum\limits_{i=1}^{j-1}k_i,\mathbb{C})\oplus \mathfrak{g}\mathfrak{l}(k_j,\mathbb{C})$.  

Moreover, by Theorem 9.2.3 of \cite{Goodman_Wallach:2009} and Littlewood-Richardson rules, Lemma 7.4 of \cite{Vogan:1986} implies 
\[\gamma_2=(\gamma_1;\beta_2^{k_2})
+(\alpha_1,\dots,\alpha_{l},0^{k_1-l};
0^{k_2-l},-\alpha_{l},\dots,-\alpha_1),\]
where $0\leq l \leq \min\{k_1,k_2\}$ and $\alpha_1\geq \dots\geq \alpha_{l}>0$. By \eqref{eta_y+1}, 
\[\begin{aligned}d=(d_1;d_2;\dots)=& \ (\alpha_1,\dots,\alpha_{l},0^{k_1-l};
0^{k_2-l},-\alpha_{l},\dots,-\alpha_1;\underbrace{\cdots}_{k-k_1-k_2\ \text{entries}})+\\ &(\underbrace{\cdots}_{k_1+k_2\ \text{entries}\ \geq 0}; \underbrace{\cdots}_{k-k_1-k_2\ \text{entries}\ \leq 0}).\end{aligned}\]
Hence
\[\delta_x \geq \alpha_x,\quad \nu_{k_2-x+1}\geq -\alpha_x, \ \forall 1\leq x\leq l,\]
and $\nu_{k_2-x+1}\geq 0$ for $x>l$. By the condition \eqref{cond_dq+1}, one get $l\geq s$, so $k_1\geq s$ and \eqref{htilde} follows immediately.

For (iii): Notice that
\begin{equation}\label{ga2_not_d}\big[E_{\gamma}^{\mathrm{U}(k)} : V\big]_{\mathrm{U}(k)} =\sum_{\tiny\begin{array}{c} \gamma_1=(\beta_1^{k_1}),\{\gamma_2,\dots,\gamma_t\}\\ \text{given in the proof of (ii)}\end{array}} \prod_{j=2}^{t}\big[E_{\gamma_{j}}^{\mathrm{U}(\sum_{i=1}^jk_i)}: \mathcal{I}_j\big]_{\mathrm{U}(\sum_{i=1}^{j}k_i)}.\end{equation}
By the assumption of $d_1,d_2$ in (iii) and the above discussion of (ii), the $\gamma_2$ contributes non-zero to the left side of \eqref{ga2_not_d} only if
\[\gamma_2=(\beta_1^{k_1};\beta_2^{k_2})+(\delta_1,\dots,\delta_s,0^{k_1-s};0^{k_2-s},-\delta_s,\dots,-\delta_1).\]
Moreover, Lemma 7.4 of \cite{Vogan:1986} implies  
\[\big[E_{\gamma_2}^{\mathrm{U}(k_1+k_2)}:\mathcal{I}(\beta_1^{k_1},\beta_2^{k_2})\big]_{\mathrm{U}(k_1+k_2)}=1.\]
Hence,
\begin{equation}\label{sum_prod_gam}
\big[E_{\gamma}^{\mathrm{U}(k)} : V\big]_{\mathrm{U}(k)} =\sum_{\tiny\begin{array}{c}\gamma_2 \ \text{as above}, \\ \gamma_3, \dots,\gamma_t=\gamma\ \text{given in (ii)}\end{array}} \prod_{j=3}^{t}\big[E_{\gamma_{j}}^{\mathrm{U}(\sum_{i=1}^jk_i)}: \mathcal{I}_j\big]_{\mathrm{U}(\sum_{i=1}^{j}k_i)}.\end{equation}

By the assumption about $d_1$ and $d_2$ in (iii), the \eqref{eta_y+1} implies 
\[\gamma_{j}=(\gamma_{j-1};\beta_{j}^{k_{j}})+(0^{k_1};\underbrace{\dots}_{k_2-s};0^s;\underbrace{\dots}_{\sum_{i=3}^{j}k_i}), \forall j\geq 3.\]
Define $\Psi_j: \mathbb{R}^{\sum_{i=1}^jk_i}\to \mathbb{R}^{(\sum_{i=2}^jk_i)-s}$ for any $j\geq 2$ by 
\[\Psi_j\big((y_1,\dots;y_{\sum_{i=1}^jk_i}\big)=(y_{k_1+1},\dots,y_{k_1+k_2-s};y_{k_1+k_2+1},\cdots,y_{\sum_{i=1}^jk_i})-(\varepsilon^{k_2-s};0^{\sum_{i=3}^jk_i}).\]
 
Apply (i) of Lemma \ref{red_tsp} to the case $m_1=\sum_{i=1}^{j-1}k_i$, $m_2=k_j$, $\eta_1=\gamma_{j-1}$, $\eta_2=(\beta_j)^{k_j}$, $\eta=\gamma_j$, $h=k_1$, $p=k_2-s$, $q=s$, one has $\forall j\geq 3$,
\begin{equation}\label{comp_psi}
\big[E_{\gamma_{j}}^{\mathrm{U}(\sum_{i=1}^jk_i)}: \mathcal{I}_j\big]_{\mathrm{U}(\sum_{i=1}^{j}k_i)}=\big[E_{\Psi_j(\gamma_{j})}^{\mathrm{U}((\sum_{i=2}^{j}k_i)-s)}: \mathcal{I}(\Psi_j(\gamma_{j-1}),\beta_j^{k_j})\big]_{\mathrm{U}((\sum_{i=2}^{j}k_i)-s)}.\end{equation}
Notice that
\begin{equation}\label{prod_gamp}\begin{aligned} & \big[E_{\Psi(\gamma)}^{\mathrm{U}((\sum_{i=2}^{j}k_i)-s)} : V'\big]_{\mathrm{U}((\sum_{i=2}^{j}k_i)-s)}\\ = & 
\sum_{\tiny\begin{array}{c} \gamma_2'=(\beta_2-\varepsilon)^{k_2-s}, \gamma_3'\in \mathbb{Z}^{k_2+k_3-s},\\ \dots,\gamma_t'=\Psi(\gamma)\ \text{are dominant of type A}\end{array}} \prod_{j=2}^{t}\big[E_{\gamma_{j}'}^{\mathrm{U}((\sum_{i=2}^{j}k_i)-s)}: \mathcal{I}(\gamma_{j-1}',\beta_j^{k_j})\big]_{\mathrm{U}((\sum_{i=2}^{j}k_i)-s)}.\end{aligned}\end{equation}
Comparing with \eqref{sum_prod_gam}, the \eqref{comp_psi} implies
\[\big[E_{\gamma}^{\mathrm{U}(k)} : V\big]_{\mathrm{U}(k)} \leq \big[E_{\Psi(\gamma)}^{\mathrm{U}((\sum_{i=2}^{j}k_i)-s)} : V'\big]_{\mathrm{U}((\sum_{i=2}^{j}k_i)-s)}.\]
On the other hand, 
let $\Phi_j: \mathbb{R}^{(\sum_{i=2}^jk_i)-s} \to \mathbb{R}^{\sum_{i=1}^jk_i}$ for any $j\geq 2$ by 
\[\begin{aligned} & \Phi_j\big((z_1,\dots,z_{k_2-s};y_1,\cdots,y_{\sum_{i=3}^jk_i})\big)\\ = &\ \big(\beta_1^{k_1}+(\delta_1,\dots,\delta_s,0^{k_1-s}); (z_1,\dots,z_{k_2-s})+\varepsilon^{k_2-s}; \beta_2^{s}+(-\delta_s,\dots,-\delta_1); y_1,\dots,y_{\sum_{i=3}^jk_i}).\end{aligned}\]
Then for any $\gamma_2',\dots,\gamma_t'$ such that the right side of \eqref{prod_gamp} is non-zero, $\Phi_j(\gamma_j')$ are dominant weight of type $A$ for $j\geq 2$, and $\Phi_2(\gamma_2')=\gamma_2$, $\Phi_t(\gamma_t')=\gamma_t$. Moreover, $\Psi_i(\Phi_i\gamma_i')=\gamma_i'$ for $i\geq 2$. By \eqref{comp_psi}, $\{\gamma_j:=\Phi_j(\gamma_j'), j\geq 2\}$ contributes to \eqref{sum_prod_gam}. Hence,
\[\big[E_{\gamma}^{\mathrm{U}(k)} : V\big]_{\mathrm{U}(k)} \geq \big[E_{\Psi(\gamma)}^{\mathrm{U}(\sum_{i=2}^{j}k_i-s)} : V'\big]_{\mathrm{U}(\sum_{i=2}^{j}k_i-s)}.\]
The lemma now follows immediately.
\end{proof}

\begin{lemma}\label{red_tsp}
\noindent (i)  Let $m_1,m_2\in \mathbb{N}_+$ and $m=m_1+m_2$. Let $\eta_1=(a_1,\dots,a_{m_1})$ (resp. $\eta_2=(b_1,\dots,b_{m_2})$) be a highest weight of a $\mathrm{U}(m_1)$-type (resp. $\mathrm{U}(m_2)$-type). Assume that $\eta$ is a highest weight of a $\mathrm{U}(m)$-type satisfying that there exist $h,p,q\in \mathbb{N}$ with $h+p+q\leq m_1$, such that  
\[\eta=(a_1,\dots,a_{h};v_1;a_{h+p+1},\dots,a_{h+p+q};v_2;v_3),\]
for some $v_1\in \mathbb{Z}^{p}$, $v_2\in \mathbb{Z}^{m_1-h-p-q}$ and $v_3\in \mathbb{Z}^{m_2}$.

For any $\varepsilon\in \mathbb{N}$ such that $a_{h+p}-\varepsilon\geq a_{h+p+q}$, set $\eta_1'=(a_{h+1},\dots,a_{h+p}; a_{h+p+q+1},\dots,a_{m_1})-(\varepsilon^{p};0^{m_1-h-p-q})$ and $\eta'=(v_1-\varepsilon^{p};v_2;v_3)$. Then  
\[\begin{aligned} & \big[E_{\eta}^{\mathrm{U}(m)} : \mathrm{Ind}_{\mathrm{U}(m_1)\times \mathrm{U}(m_2)}^{\mathrm{U}(m)}\big(E_{\eta_1}^{\mathrm{U}(m_1)}\boxtimes E_{\eta_2}^{\mathrm{U}(m_2)}\big)\big]_{\mathrm{U}(m)}\\ = &\ 
\big[E_{\eta'}^{\mathrm{U}(m-h-q)} : \mathrm{Ind}_{\mathrm{U}(m_1-h-q)\times \mathrm{U}(m_2)}^{\mathrm{U}(m-h-q)}\big(E_{\eta_1'}^{\mathrm{U}(m_1-h-q)}\boxtimes E_{\eta_2}^{\mathrm{U}(m_2)}\big)\big]_{\mathrm{U}(m-h-q)}.\end{aligned}\]
\noindent (ii) Let $m\in \mathbb{N}_+$, and let $\omega_1=(a_1,\dots,a_m)$, $\omega_2=(b_1,\dots,b_m)$ be two highest weights of $\mathrm{U}(m)$-types. Assume that $\omega$ is a highest weight of a $\mathrm{U}(m)$-type satisfies that there exist $h,p,q\in \mathbb{N}$ with $h+p+q\leq m$, such that  
\[\omega=(a_1,\dots,a_{h};v_1;a_{h+p+1},\dots,a_{h+p+q};v_2),\]
for some $v_1\in \mathbb{Z}^{p}$, $v_2\in \mathbb{Z}^{m-h-p-q}$. And $\omega_2$ satisfies $b_{m-x+1}=0$ for $1\leq x\leq h+q$.

For any $\varepsilon\in \mathbb{N}$ such that $a_{h+p}-\varepsilon\geq a_{h+p+q}$, set $\omega_1'=(a_{h+1},\dots,a_{h+p};a_{h+p+q+1},\dots,a_m)-(\varepsilon^p;0^{m-h-p-q})$, $\omega_2'=(b_1,\dots,b_{m-h-q})$, and $\omega'=(v_1;v_2)-(\varepsilon^p;0^{m-h-p-q})$. 
Then 
\[\begin{aligned} & \big[E_{\omega}^{\mathrm{U}(m)} :   E_{\omega_1}^{\mathrm{U}(m)}\otimes E_{\omega_2}^{\mathrm{U}(m)} \big]_{\mathrm{U}(m)}\\ = &\ 
\big[E_{\omega'}^{\mathrm{U}(m-h-q)} :  E_{\omega_1'}^{\mathrm{U}(m-h-q)}\otimes E_{\omega_2'}^{\mathrm{U}(m-h-q)} \big]_{\mathrm{U}(m-h-q)}.\end{aligned}\]  
\end{lemma}
\begin{proof}
By Theorem 9.2.3 of \cite{Goodman_Wallach:2009}, one can reduce (i) to (ii). And (ii) can be check directly by the Littlewood-Richardson rules. 
\end{proof}

Let's prove that $\pi$ contains a unique and multiplicity-free spin lowest $K$-type, which satisfies \eqref{slkt-cond}, by showing Theorem \ref{uni_m1_GLnH}. By the paragraph before Proposition \ref{FS-sc}, one can assume $\pi$ is a FS-scattered representation.

\begin{theorem}\label{uni_m1_GLnH} There is exactly one element $\tau$ in
\[
\Omega_0:=\{w(\Lambda-\rho_c)+\rho_n\ |\ w\in W(\mathfrak{k},\mathfrak{t})\ \text{s.t.}\ w(\Lambda-\rho_c)+\rho_n\ \text{is}\ \Delta^+(\mathfrak{k},\mathfrak{t})\text{-dominant}\}  \]
such that
\begin{equation}\label{cup>0}
\big[(Z^{\#}\boxtimes \mathbf{1})\otimes S(\mathfrak{u}\cap\mathfrak{p}):H^i(\mathfrak{u}\cap \mathfrak{k},E_{\tau})\big]_{L\cap K}>0
\end{equation}
for some $i\in \mathbb{N}$. When condition \eqref{cup>0} holds, then $i=0$ and the multiplicity \eqref{cup>0} is one.
\end{theorem}

\begin{proof} 
\noindent {\bf Step 1:} Notice that $H^0(\mathfrak{u}\cap \mathfrak{k},E_{\tau})=E_{\tau}$, it suffices to show
\[\big[(Z^{\#}\boxtimes \mathbf{1})\otimes S(\mathfrak{u}\cap \mathfrak{p}): \mathop{\oplus}\limits_{\tau\in \Omega_0}\mathop{\oplus}\limits_{i\in \mathbb{N}} H^i(\mathfrak{u}\cap \mathfrak{k},E_{\tau})\big]_{L\cap K}=\big[(Z^{\#}\boxtimes \mathbf{1})\otimes S(\mathfrak{u}\cap \mathfrak{p}): \mathop{\oplus}\limits_{\tau\in \Omega_0}E_{\tau}\big]_{L\cap K}=1.\]
Let us approximate the sets of  $L\cap K$-types in $(Z^{\#}\boxtimes \mathbf{1})\otimes S(\mathfrak{u}\cap \mathfrak{p})$ and $\mathop{\oplus}\limits_{\tau\in \Omega_0}\mathop{\oplus}\limits_{i\in \mathbb{N}} H^i(\mathfrak{u}\cap \mathfrak{k},E_{\tau})$: 
Define
\[\mathcal{A}=\text{the set of highest weights of }\ L\cap K\text{-types in} \ (Z^{\#}\boxtimes \mathbf{1})\otimes
(\mathcal{E}_1 \oplus
\mathcal{E}_2),\]
where $\mathcal{E}_1 \oplus
\mathcal{E}_2$ contains all the $L\cap K$-types of $S(\mathfrak{u}\cap\mathfrak{p})$ by the Lemma \ref{sup_glnh}. Hence, $\mathcal{A}$ contains all the $L\cap K$-types in $(Z^{\#}\boxtimes \mathbf{1})\otimes S(\mathfrak{u}\cap\mathfrak{p})$. 

Set 
\begin{equation}\label{cone}
\mathcal{N}=\{(x_1,\dots,x_k,y_1,\dots,y_r)\ |\ x_i\leq 0, 1\leq i\leq k; y_j\geq 0, 1\leq j\leq r\}.\end{equation}
Define
\[ \mathcal{B}= \{\tau+\xi \ |\ \tau\in \Omega_0, \xi\in \mathcal{N}\}. \] 
By \eqref{coh_lk}, all highest weights of $L\cap K$-types in $\mathop{\oplus}\limits_{i\in \mathbb{N}}H^i(\mathfrak{u}\cap \mathfrak{k},E_{\tau})$ are contained in $\{w_1(\tau+\rho_c)-\rho_c\ |\ w_1\in W^1\}$. Since $\rho_c$ and $\tau$ are both $\Delta^+(\mathfrak{k},\mathfrak{t})$ dominant, one has that for any $w_1\in W^1$,
\[w_1\rho_c-\rho_c\in \mathcal{N},\ \text{and}\ w_1\tau-\tau\in \mathcal{N}.\]
Therefore, $w_1(\tau+\rho_c)-\rho_c=\tau+(w_1\tau-\tau)+(w_1\rho_c-\rho_c)\in \tau+\mathcal{N}$. 
Hence, $\mathcal{B}$ contains 
 all the highest weights of $L\cap K$-types in $\mathop{\oplus}\limits_{\tau\in \Omega_0}\mathop{\oplus}\limits_{i\in \mathbb{N}} H^i(\mathfrak{u}\cap \mathfrak{k},E_{\tau})$. 

\noindent {\bf Step 2:} Let's give an approximations of the set $\mathcal{A}\cap \mathcal{B}$. 

Let

$$\zeta_{\pi}:=(\mathcal{L}_1;\dots;\mathcal{L}_t;\mathcal{T}_u),$$
where $\mathcal{L}_1,\dots,\mathcal{L}_t,\mathcal{T}_u$ are defined in Theorem \ref{thm-main}. Define a function
\[\phi:\mathbb{R}^n\to \mathbb{R}, (x_1,\dots,x_{k},y_1,\dots,y_{r})\mapsto \sum_{i=1}^{k} x_i-\sum_{i=1}^r y_i.\]
Let's show that $\phi(v)\geq \phi(\zeta_{\pi}), \forall\ v\in \mathcal{A}$, and $\phi(v)\leq \phi(\zeta_{\pi}), \forall\ v\in \mathcal{B}$.
 
For any $v\in \mathcal{A}$, by definition, assume $v$ is highest weight of a $L\cap K$-type in $(E_{\gamma}\boxtimes \mathbf{1})\otimes (E_{\chi_1}\boxtimes E_{\chi_2})$, where $E_{\gamma}$ is an $L\cap K$-type in $Z^{\#}$ and $E_{\chi_1}\boxtimes E_{\chi_2}$  is an $L\cap K$-type in $\mathcal{E}_1\oplus\mathcal{E}_2$. By Lemma \ref{high_lem}, $v$ is equal to $(\gamma;\chi_2)$ plus a weight in $E_{\chi_1}\boxtimes 1$.  Since all weights in $E_{\chi_1}$ have the same sum of the coordinates, one has
\[\phi(v)=\phi\big((\gamma;\chi_2)\big)+\phi\big((\chi_1;0^r)\big)=\phi\big((\gamma;0^{r})\big)+\phi\big((\chi_1;\chi_2)\big),\ \forall v\in \mathcal{A}.\]
By definition of $\mathcal{E}_1$ and $\mathcal{E}_2$, one has $\phi\big((\chi_1;\chi_2)\big)\geq 0$ with the equality happens if and only if $(E_{\chi_1}\boxtimes E_{\chi_2})$ comes from $\mathcal{E}_1$.

By (i) of Lemma \ref{kt_of_zs}, we know any $L\cap K$-type $\gamma$ of $(Z^{\#}\boxtimes \mathbf{1})$ is in the set
\[\big\{
(\mathcal{L}_1,\dots,\mathcal{L}_t)+d\ |\
 d\in \mathbb{N}\langle\Delta(\mathfrak{u}_0)\rangle\big\}.\]
Hence, $\phi\big((\gamma;0)\big)=\phi(\zeta_{\pi})$. Combining with above discussion, one has
\begin{equation}\label{phi_a}
\phi(v)\geq \phi(\zeta_{\pi}),\ \forall v\in \mathcal{A}.
\end{equation}

For any $v=\tau+\xi \in \mathcal{B}$ with $\tau=w(\Lambda-\rho_c)+\rho_n\in \Omega_0$, notice that
\[\phi(\tau+\xi) \leq \phi(\tau)=\phi\big(w(\Lambda-\rho_c)\big)+\phi(\rho_n),\]
where the first equality happens if and only if $\xi=0$.

Let
\[S_0:\mathbb{R}^n\to \mathbb{R}, \ (z_1,\dots,z_{n})\mapsto \sum_{i=1}^{n}z_i.\]
Since $\Lambda-\rho_c$ is $\Delta^+(\mathfrak{k},\mathfrak{t})$ dominant, one has $\phi(w(\Lambda-\rho_c))\leq S_0(\Lambda-\rho_c)$, and the equality happens if and only if
\[w(\Lambda-\rho_c)\in -\mathcal{N}.\] 
A direct calculation shows $S_0(\Lambda-\rho_c)+\phi(\rho_n)=\phi(\zeta_{\pi})$,
so one has
\begin{equation}\label{phi_b_R}
\phi(\tau+\xi) \leq \phi(\zeta_{\pi}),\ \forall \tau+\xi \in \mathcal{B}.
\end{equation}

Combining \eqref{phi_a} and \eqref{phi_b_R}, we know that any element $v\in  \mathcal{A}\cap \mathcal{B}$ must satisfy $\phi(v)=\phi(\zeta_{\pi})$. By the condition on when equalities in \eqref{phi_a} and \eqref{phi_b_R} hold, such $v$ is a highest weights of an $L\cap K$-type in $(Z^{\#}\boxtimes \mathbf{1})\otimes \mathcal{E}_1$, and is equal to some $w(\Lambda-\rho_c)+\rho_n\in \Omega_0$ such that $w(\Lambda-\rho_c)\in -\mathcal{N}$. 

By Lemma \ref{sup_glnh}, all the $L\cap K$-types in $\mathcal{E}_1$ has multiplicity one in $S(\mathfrak{u}\cap \mathfrak{p})$. Therefore, 
\[\big[(Z^{\#}\boxtimes \mathbf{1})\otimes S(\mathfrak{u}\cap \mathfrak{p}): \mathop{\oplus}\limits_{\tau\in \Omega_0}\mathop{\oplus}\limits_{i\in \mathbb{N}} H^i(\mathfrak{u}\cap \mathfrak{k},E_{\tau})\big]_{L\cap K}=\big[(Z^{\#}\boxtimes \mathbf{1})\otimes \mathcal{E}_1: \mathop{\oplus}\limits_{\tau\in \Omega}  E_{\tau}\big]_{L\cap K},\]
where $\Omega$ is defined by
\[ \Omega= \{w(\Lambda-\rho_c)+\rho_n\ |\ w(\Lambda-\rho_c)+\rho_n\in \Omega_0,\ w(\Lambda-\rho_c)\in -\mathcal{N}\}. \]

By the following Proposition \ref{prop_ktype}, take $V=Z^{\#}$, i.e. let $\beta_i=b_i+k_i-1$, one has 
\[\big[(Z^{\#}\boxtimes \mathbf{1})\otimes \mathcal{E}_1: \mathop{\oplus}\limits_{\tau\in \Omega}  E_{\tau}\big]_{L\cap K}=1.\]
Now the theorem follows immediately. 
\end{proof}

To finish the proof above, we prove the following Proposition \ref{prop_ktype}, whose statement is  independent of the theorem above.  

In the following proof, let ``$>$" denote the lexicographical order over $\mathbb{R}^m$ for any $m\in \mathbb{N}_+$: $(x_1,\dots,x_m)>(y_1,\dots,y_m)$ if and only if there exists $1\leq i\leq m$  such that $x_j=y_j$ for all $j<i$, and $x_i>y_i$.

Let us use ``$\succ$" to denote the relation : $(x_1,\dots,x_m) \succ (y_1,\dots,y_m)$ if and only if $x_i> y_i$ for all $1\leq i\leq m$.

\begin{proposition}\label{prop_ktype}
Given the datum ``$\zeta=\big(\beta_1^{k_1};\dots;\beta_t^{k_t};0^r\big)$ with $r\in \mathbb{N}_+$, $k_i\in \mathbb{N}_+$ and $\beta_i\in \mathbb{N}$ for $1\leq i\leq t$", which are equivalent to the datum of chains:
\[\mathcal{A}_i=[\beta_i+k_i,\beta_i-k_i+2],\ 1\leq i\leq t;\  \mathcal{U}=[2r-1,1].\]
Let $k=\sum\limits_{i=1}^t k_i\ \text{and}\ n=k+r$.
Assume that $\Lambda=\mathop{\cup}\limits_i^t \mathcal{A}_i\cup \mathcal{U}$ is $u$-interlaced as Definition \ref{chain}. 

Let $V=\mathrm{Ind}_{\prod_{i=1}^t\mathrm{GL}(k_i,\mathbb{C})}^{\mathrm{GL}(k,\mathbb{C})}\big(
\mathop{\boxtimes}\limits_{i=1}^t\mathbb{C}_{\beta_i}\big)$ be the $\mathrm{GL}(k,\mathbb{C})$-module as Lemma \ref{kt_of_zs}. Define
\[\mathcal{H}= \  \text{All highest weights of $\mathrm{U}(k)\times \mathrm{Sp}(r)$-types in} \ (V\boxtimes \mathbf{1})\otimes \mathcal{E}_1,\] 
\[\Omega=\left\{w(\Lambda-\rho_c)+\rho_n\ \big\vert\ w\in W(C_n)\ s.t.\ \Big\{\begin{array}{l}w(\Lambda-\rho_c)+\rho_n\ \text{is dominant of type C}\\ w(\Lambda-\rho_c)\in -\mathcal{N}\end{array}\right\}, \]
where $\mathcal{E}_1$ is defined in Lemma \ref{kt_of_zs}, $W(C_n)$ is the Weyl group of type $C_n$, and $\mathcal{N}$ is defined in \eqref{cone}. 

Then $\mathcal{H}\cap \Omega$ contains exactly one element, denoted by $\tau$, and $E_{\tau}^{\mathrm{U}(k)\times \mathrm{Sp}(r)}$ shows up with multiplicity one in $(V\boxtimes \mathbf{1})\otimes \mathcal{E}_1$.
\end{proposition}

\begin{proof} 
By definition of $\mathcal{E}_1$, one can see that any element in $\mathcal{H}$ is of the form
\begin{equation}\label{C=gam+h}
(\gamma;0^r)+(h_1,\dots,h_k;a_1,\dots,a_x,0^{r-x}),
\end{equation} 
where $\gamma$ is the highest weight of a $\mathrm{U}(k)$-type of $V$, and $(h_1,\dots,h_k)$ is a weight of $E^{\mathrm{U}(k)}_{(a_1,\dots,a_x,0,\dots,0)}$ for some $a_1\geq \dots\geq a_x> 0$, $x\leq \min\{k,r\}$. Therefore, there are at least $x$ non-zero entries among $h_1,\dots,h_k$, and 
\begin{equation}\label{h<=a}  
h_j\geq 0,\ \forall\ 1\leq j\leq k;\ \text{and}\  \sum_{1\leq j\leq i}h_{\psi(j)}\leq \sum_{1\leq j\leq i}a_j, \forall\ \psi\in W(A_{k-1}),  \forall\ 1\leq i\leq k,\end{equation}
where $W(A_{k-1})$ is the symmetric group of $\{1,\dots,k\}$, and set $a_i=0$ when $x< i\leq k$.

Moreover, Lemma \ref{kt_of_zs} gives that 
\begin{equation}\label{gam=b+d}
\gamma=(\beta_1^{k_1};\dots;\beta_t^{k_t})+d,\ 
 d\in \mathbb{N}\langle\Delta(\mathfrak{u}_0)\rangle. \end{equation}
 
Define $\mathcal{Q}_1: \mathbb{R}^n\to \mathbb{R}^k$ by  \[\mathcal{Q}_1\big((u;y)\big)=u, \forall u\in \mathbb{R}^{k}, \forall y\in \mathbb{R}^{r}.\] 

Define $\mathcal{Q}_2: \mathbb{R}^{n-r+\min\{k,x\}}\times 0^{r-\min\{k,x\}} (\subset \mathbb{R}^n) \to \mathbb{R}^k$ by 
\[\mathcal{Q}_2\big((u;(y_1,\dots,y_x,0^{r-x}))\big)=(y_1,\dots,y_x,0^{k-x}), \forall u\in \mathbb{R}^{k}, \forall 0\leq x\leq \min\{k,r\}.\]

By the property of $\mathcal{E}_1$, for any $v=(u_1,\dots,u_k;a_1,\dots,a_x,0^{r-x})\in \mathbb{Z}^n$ with $u_1\geq \dots\geq u_k$, $a_1\geq \dots\geq a_x\geq 0$ and $x\leq \min\{k,r\}$, the multiplicity of $E_{v}^{\mathrm{U}(k)\times \mathrm{Sp}(r)}$ in $(V\boxtimes \mathbf{1})\otimes \mathcal{E}_1 $ is equal to
\begin{equation}\label{mul_prod} \sum_{E_{\gamma}^{\mathrm{U}(k)}}\big[E_{\gamma}^{\mathrm{U}(k)}:V|_{\mathrm{U}(k)}\big]\cdot \big[E_{\gamma}^{\mathrm{U}(k)}\otimes E_{\mathcal{Q}_2(v)}^{\mathrm{U}(k)} : E_{\mathcal{Q}_1(v)}^{\mathrm{U}(k)} \big]. \end{equation}  
The proof of the theorem are divided into the following three cases.

\textbf{Case (I)}. This is the basic case and the proof gives the idea of the case (II) and (III). Assume that the chains are interlaced as follows:
\[\begin{array}{c} (\mathcal{A}_1)\ (\mathcal{A}_2)\ \dots (\mathcal{A}_t)\\
(   \qquad \qquad
 \mathcal{U}\qquad \qquad)\end{array},\]
i.e. $2r-1> \beta_1+k_1\geq \beta_1-k_1+2> \dots> \beta_t+k_t\geq \beta_t-k_t+2>1$.  
Then 
\[\begin{aligned}\Lambda= &\ ([2r-1,\beta_1+k_1+1];\{\beta_1+k_1,\beta_1-k_1+2\};\\ & [\beta_1-k_1+3,\beta_2+k_2+1];\{\beta_2+k_2,\beta_2-k_2+2\};\dots;
[\beta_t+k_t,\beta_t-k_t+2];\{\beta_t-k_t+1,1\}).\end{aligned}\]
Recall that $\rho_c=(n,n-1,\dots,1)$, so
$$
\Lambda-\rho_c=(\{q,p_1\};p_1^{2k_1-1};\{p_1,p_2\};p_2^{2k_2-1};\{p_2,p_3\};
\dots;p_t^{2k_t-1};\{p_t,0\}),$$
for $q=r-1-k$, $p_i=\frac{\beta_i-k_i}{2}-\sum\limits_{j=i+1}^{t}k_j$. 
To prove the theorem in this case, it suffices to prove that for any $w(\Lambda-\rho_c)+\rho_n\in \mathcal{H}\cap \Omega$, one has
\begin{equation}\label{i_uniq_wl-r}
w(\Lambda-\rho_c)=(p_1^{k_1};\dots,p_t^{k_t};(-p_t)^{k_t};\dots;(-p_1)^{k_1};\{-q,0\}),\end{equation}
and $E_{w(\Lambda-\rho_c)+\rho_n}^{\mathrm{U}(k)\times \mathrm{Sp}(r)}$ shows up in $(V\boxtimes \mathbf{1})\otimes \mathcal{E}_1$ with multiplicity one.

To prove \eqref{i_uniq_wl-r}, for any $v=(v_1,\dots,v_t,v_t',\dots,v_1',v_0)\in \mathbb{R}^n$ with $v_i, v_i'\in \mathbb{R}^{k_i}$ and ${v_0}\in \mathbb{R}^{q+1}$, let 
\[S_0(v)=||{v_0}||_1,\ S_i(v)=\sum {v_i},\ S_i'(v)=\sum v_i',\] 
where $||\cdot||_1$ (resp. $\sum$ ) means summing up the absolute of the coordinates (resp. summing up the coordinates) of the vector.    

Let's show that any $w(\Lambda-\rho_c)+\rho_n\in \mathcal{H}\cap \Omega$ must satisfy \eqref{i_uniq_wl-r}, it suffices to show that: Given any
$v\in \mathbb{R}^n$, consider the vector 
\begin{equation}\label{S0t}\big(S_0(v),S_t(v-\zeta)-S_t'(v-\zeta),\dots,S_1(v-\zeta)-S_1'(v-\zeta)\big)\in \mathbb{R}^{t+1}
\end{equation}
\begin{itemize}
\item[(i)] If $v\in \mathcal{H}$, one has \eqref{S0t} $\leq 0^{t+1}$ under the lexicographical order.
\item[(ii)] If
$v=w(\Lambda-\rho_c)+\rho_n\in \Omega$, one has \eqref{S0t} $\geq 0^{t+1}$ under the lexicographical order, and the equality happens if and only if \eqref{i_uniq_wl-r} holds.
\end{itemize}

For (i): Given any $v\in \mathcal{H}$. By \eqref{C=gam+h}, $v$ contains at most $k$ positive elements in the last $r$ entries, so the last $r-k=q+1$ entries of $v$ are all zero. Hence $S_0(v)=0$.
 
By \eqref{C=gam+h} and \eqref{gam=b+d}, one has for any $1\leq j\leq t$,
\[S_j(v-\zeta)=S_j\big((d;0^r)\big)+S_j\big((h_1,\dots,h_k;0^r)\big),\  S_j'(v-\zeta)=S_j'\big((0^k;a_1,\dots,a_k,0^{r-k})\big).\]
For any $1\leq i\leq t$, since $d\in \mathbb{N}\langle\Delta(\mathfrak{u}_0)\rangle$, $\sum_{j=i}^tS_j\big((d;0^r)\big)\leq 0$;  by \eqref{h<=a}, 
\[\sum_{j=i}^tS_j\big((h_1,\dots,h_k;0^r)\big)\leq \sum_{j=i}^tS_j'\big((0^k;a_1,\dots,a_k,0^{r-k})\big).\] 
Hence,
$$
\sum_{j=i}^t S_j(v-\zeta)\leq \sum_{j=i}^t S_j'(v-\zeta),\ \forall 1\leq i\leq t,
$$
which implies (i) immediately.

For (ii): Given any
$v=w(\Lambda-\rho_c)+\rho_n\in \Omega$. Since $\rho_n=(n-1,n-2,\dots,1,0)$, one has $S_0(v)\geq 0$ with equality holds if and only if 
\begin{equation}\label{w_{q+1}}w(\Lambda-\rho_c)=(\underbrace{\dots}_{k};\underbrace{\dots}_{k};\{-q,0\}).\end{equation}

To show (ii), one can assume $S_0(v)=0$, and so \eqref{w_{q+1}} holds. We claim that $S_t(v-\zeta)\geq S_t'(v-\zeta)$.

By definition of $\Omega$ and $\mathcal{N}$, one knows that the $(k-k_t+1)^{st}$ through $k^{th}$ entries of $w(\Lambda-\rho_c)$ are all non-negative. Combining with \eqref{w_{q+1}}, one has  
\begin{equation}\label{s1>}
\begin{aligned}S_t\big(w(\Lambda-\rho_c)+\rho_n-\zeta\big)= &\ S_t\big(w(\Lambda-\rho_c)\big)+S_t(\rho_n-\zeta)\geq p_tk_t+S_t(\rho_n-\zeta),\end{aligned}
\end{equation}
and the equality happens if and only if
\begin{equation}\label{Bt}
w(\Lambda-\rho_c)=(\underbrace{\dots}_{k-k_t},p_t^{k_t};\underbrace{\dots}_{r-q-1};\{-q,0\}).
\end{equation}

Again by definition of $\Omega$ and $\mathcal{N}$, the last $r$ entries of $w(\Lambda-\rho_c)$ are non-positive. Combining with \eqref{w_{q+1}}, one has
\begin{equation}\label{s2<}
S_t'(w(\Lambda-\rho_c)+\rho_n-\zeta)=S_t'\big(w(\Lambda-\rho_c)\big)+S_t'(\rho_n-\zeta)\leq -p_tk_t+S_t'(\rho_n-\zeta),
\end{equation}
and the equality happens if and only if
\begin{equation}\label{Bt'}
w(\Lambda-\rho_c)=(\underbrace{\dots}_{k};(-p_t)^{k_t};\underbrace{\dots}_{r-q-k_t-1};\{-q,0\}).\end{equation}

Notice that $S_t(\rho_n)=S_t'(\rho_n)+k_t^2$ and $S_t(\zeta)=S_t'(\zeta)+\beta_tk_t$. Combining with $\beta_t=2p_t+k_t$, one has
\[p_tk_t+S_t(\rho_n-\zeta)=-p_tk_t+S_t'(\rho_n-\zeta).\] 
Hence, \eqref{s1>} and \eqref{s2<} imply $S_t(v-\zeta)-S_t'(v-\zeta)\geq 0$, and the equality happens if and only if \eqref{Bt} and \eqref{Bt'} hold.

To show (ii), one can assume $S_t(v)=S_t'(v)$ and so \eqref{Bt} and \eqref{Bt'} hold. By repeating the arguments $t-1$ times to $S_{t-1}(v)-S_{t-1}'(v)$, $\dots$, $S_1(v)-S_1'(v)$, one can get (ii), and so \eqref{i_uniq_wl-r}. 

It remains to prove that the $E_{\tau}^{\mathrm{U}(k)\times \mathrm{Sp}(r)}$ with $\tau=w(\Lambda-\rho_c)+\rho_n$ satisfying \eqref{i_uniq_wl-r}, shows up in $(V\boxtimes \mathbf{1})\otimes \mathcal{E}_1$ with multiplicity one. By \eqref{i_uniq_wl-r}, one has
\begin{equation}\label{i_tau}
\begin{array}{rl}\tau = & (\zeta_{V};0^r)+(\{r-\frac{\beta_1-k_1+2}{2},r-\frac{\beta_1+k_1}{2}\};\dots;\{r-\frac{\beta_t-k_t+2}{2},r-\frac{\beta_t+k_t}{2}\}; \vspace{0.8em} \\ &\{r-\frac{\beta_t-k_t+2}{2},r-\frac{\beta_t+k_t}{2}\};\dots;\{r-\frac{\beta_1-k_1+2}{2},r-\frac{\beta_1+k_1}{2}\};0^{q+1}).\end{array}\end{equation}
Assume that 
$\tau=(\gamma;0^r)+(h_1,\dots,h_k;a_1,\dots,a_x,0^{r-x})$ as \eqref{C=gam+h}, then 
\[\begin{array}{c}(a_1,\dots,a_x,0^{r-x})=(\{r-\frac{\beta_t-k_t+2}{2},r-\frac{\beta_t+k_t}{2}\};\dots;\{r-\frac{\beta_1-k_1+2}{2},r-\frac{\beta_1+k_1}{2}\};0^{q+1})\end{array}.\]
Combining with $\gamma$ satisfying \eqref{gam=b+d}, the \eqref{i_tau} implies that
\begin{equation}\label{hbd}(h_1,\dots,h_k)=(\{r-\frac{\beta_1-k_1+2}{2},r-\frac{\beta_1+k_1}{2}\};\dots;\{r-\frac{\beta_t-k_t+2}{2},r-\frac{\beta_t+k_t}{2}\})-d,\end{equation}
where $d\in \mathbb{N}\langle\Delta(\mathfrak{u}_0)\rangle$ as \eqref{gam=b+d}.
Write $d=(d_1;\dots;d_t)$ with $d_i\in \mathbb{Z}^{k_i}$, then one has $d_t\preceq 0^{k_t}$.  Let us show $d=0^k$: Since $h_1,\dots,h_k$ satisfy \eqref{h<=a}, one has \[\sum\limits_{j=1}^{k_t}a_j-\sum\limits_{j=k-k_t+1}^{k}h_j \geq 0,\]
Combining with \eqref{hbd}, one has $\sum d_t\geq 0$, where $\sum d_t$ means summing up the coordinates of $d_t$. Therefore, one gets $d_t=0^{k_t}$. By repeating the above arguments $t-1$ times to $d_{t-1},\dots,d_1$, one gets $d=0$. Hence, $\gamma=\zeta_V$. 

Notice that $E_{\zeta_V}^{\mathrm{U}(k)}$ has multiplicity one in $V$. Since the weight $\mathcal{Q}_1(\tau)-\zeta_V=\big(\{r-\frac{\beta_1-k_1+2}{2},r-\frac{\beta_1+k_1}{2}\};\dots;\{r-\frac{\beta_t-k_t+2}{2},r-\frac{\beta_t+k_t}{2}\}\big)$ has multiplicity one in $E_{\mathcal{Q}_2(\tau)}^{\mathrm{U}(k)}=E_{(a_1,\dots,a_x,0^{k-x})}^{\mathrm{U}(k)}$, one has 
\[\big[E_{\zeta_V}^{\mathrm{U}(k)}\otimes E_{\mathcal{Q}_2(\tau)}^{\mathrm{U}(k)} : E_{\mathcal{Q}_1(\tau)}^{\mathrm{U}(k)} \big]\leq 1.\] 
Moreover, the equality holds by Littlewood-Richarson rules, see also Lemma 3.7 of \cite{Dong_Wong_1:2022}. Hence, \eqref{mul_prod} implies that the multiplicity of $E_{\tau}^{\mathrm{U}(k)\times \mathrm{Sp}(r)}$ in $(V\boxtimes \mathbf{1})\otimes \mathcal{E}_1$ is one. 

This finishes the proof of case (I).

\textbf{Case (II)}. Assume that the chains are as follows:
\[\begin{array}{rl} (\quad \mathcal{A}_1\quad )  &(\mathcal{A}_2) \  \dots \  ( \mathcal{A}_t)\\
(\qquad & \qquad
 \mathcal{U}  \qquad  \qquad \quad).\end{array}\] 
i.e. $\beta_1+k_1 > 2r-1> \beta_1-k_1+2> \dots> \beta_t+k_t\geq \beta_t-k_t+2>1$.  

As case (I), one has   
\[\Lambda-\rho_c=(\{q,p_1\};p_1^{2m};\{p_1,p_2\};p_2^{2k_2-1};\{p_2,p_3\};\dots;p_t^{2k_t-1};\{p_t,0\}),\]
for $q=\beta_1+k_1-n$, $p_i=\frac{\beta_i-k_i}{2}-\sum\limits_{j=i+1}^{t}k_j$, $m=r-1-\frac{\beta_1-k_1}{2}$.  

To prove the case (II), it suffices to show that for any $w(\Lambda-\rho_c)+\rho_n\in \mathcal{H}\cap \Omega$, one has
\begin{equation}\label{wl-rII}
w(\Lambda-\rho_c)= 
(p_1^{m};\{p_1,q\};p_2^{k_2};\dots;p_t^{k_t};(-p_t)^{k_t};\cdots;(-p_2)^{k_2};(-p_1)^{m};\{-p_1,0\}),\end{equation}
And $E_{w(\Lambda-\rho_c)+\rho_n}^{\mathrm{U}(k)\times \mathrm{Sp}(r)}$ has multiplicity one in $(V\boxtimes \mathbf{1})\otimes \mathcal{E}_1$.

Similar to case (I), for any 
\[v=({v_1},{v_+},{v_2},\dots,{v_t},v_t',\dots,v_1',v_0)\in \mathbb{R}^n\] 
with ${v_+}\in \mathbb{R}^{q-p_1+1}$, $v_1,v_1'\in \mathbb{R}^{m}$, $v_i, v_i'\in \mathbb{R}^{k_i}$ for $i\geq 2$, and ${v_0}\in \mathbb{R}^{p_1+1}$, let 
\[S_+(v)=\min \{{v_+}\},  \ S_0(v)=||{v_0}||_1,\  S_i(v)=\sum v_i,\  S_i'(v)=\sum v_i',\] 
where $\min$ means taking the minimal of the coordinates of the vector.  

By \eqref{C=gam+h} and \eqref{gam=b+d}, for any $v\in \mathcal{H}$, the first $k_1$ entries of the elements in $v-\zeta$ are non-negative. Hence,  
\begin{equation}\label{s+>0}
S_+(v-\zeta)\geq 0,\ \forall v\in \mathcal{H}.
\end{equation}

On the other hand,  for $w(\Lambda-\rho_c)+\rho_n\in \Omega$, one has 
\[w(\Lambda-\rho_c)+\rho_n-\zeta=w(\Lambda-\rho_c)+(\underbrace{\dots}_{m};\{-p_1,-q\};\underbrace{\dots}_{n-k_1}). \]
Therefore, $S_+\big(w(\Lambda-\rho_c)+\rho_n-\zeta\big)\geq 0$ if and only if 
\begin{equation}\label{wl-r_c-k_1}w(\Lambda-\rho_c)=(\underbrace{\dots}_{m};\{p_1,q\};\underbrace{\dots}_{n-k_1}).\end{equation} Otherwise, if $w(\Lambda-\rho_c)$ doesn't satisfy \eqref{wl-r_c-k_1}, then the $(m+1)^{st}$ through $k_1^{st}$ entries of $w(\Lambda-\rho_c)+\rho_n-\zeta$ will contain negative entries. Combining with \eqref{s+>0}, one has that for any $w(\Lambda-\rho_c)+\rho_n\in \mathcal{H}\cap\Omega$, \eqref{wl-r_c-k_1} holds.

Apply the similar argument to case (I) to $S_0(v-\zeta)$, $S_t(v-\zeta)-S_t'(v-\zeta)$, $\dots$, $S_1(v-\zeta)-S_1'(v-\zeta)$, for $v\in \mathcal{H}\cap \Omega$, one can get \eqref{wl-rII}.

Similar to the case (I), one can show that the multiplicity of $E_{w(\Lambda-\rho_c)+\rho_n}^{\mathrm{U}(k)\times \mathrm{Sp}(r)}$ in $(V\boxtimes \mathbf{1})\otimes \mathcal{E}_1$ is one. This finishes the proof of case (II).
  
\textbf{Case (III)}. In general, let's prove by induction on $t$. Generally, the chains are interlaced as follows:
\[\text{sub-case (a).}\ \begin{aligned}&(\quad \mathcal{A}_{1} \quad) \dots   \\  & \quad (\quad \mathcal{A}_{2} \quad)     \dots ,\end{aligned}\qquad \text{or} \qquad 
\text{sub-case (b).}\ \begin{aligned}&(\quad \mathcal{A}_{1} \quad )\dots   \\  & \quad (\mathcal{A}_{2})      \dots ,\end{aligned}\]
i.e. $\beta_1+k_1 >\beta_2+k_2>\beta_1-k_1+2>\beta_2-k_2+2$ in sub-case (a), or $\beta_1+k_1 >\beta_2+k_2\geq \beta_2-k_2+2>\beta_1-k_1+2$ in sub-case (b).

We will only prove the sub-case (a) in detail, omitting the sub-case (b) which can be proved similarly.

Assume the chains are as sub-case (a). Then 
\[\Lambda-\rho_c =(\{q,p_1\};  p_1^{2m}; \dots ),\]
where $q=\beta_1+k_1-n$, $p_1=\frac{\beta_1+\beta_2+k_1+k_2+1}{2}-n$ and 
 $m=\frac{\beta_2+k_2-\beta_1+k_1-1}{2}$. Note that $q-p_1+m+1=k_1$.

Assume that $\tau=w(\Lambda-\rho_c)+\rho_n\in \mathcal{H}\cap \Omega$, let's show 
\begin{equation}\label{w_uniq} \begin{aligned}w(\Lambda-\rho_c)=
&(p_1^{m};\{p_1,q\};\underbrace{\dots}_{k_2-m}; p_1^{m};\underbrace{\dots}_{n-k_1-k_2}).\end{aligned}\end{equation}

\textbf{Step 1}. Since $\tau\in \mathcal{H}$, by \eqref{C=gam+h} and \eqref{gam=b+d}, the first $k_1$ entries of
$\tau-\zeta$ are non-negative. For similar reason to case (II), by $w(\Lambda-\rho_c)+\rho_n\in \mathcal{H}$, one has
\begin{equation}\label{1wl-r}
w(\Lambda-\rho_c)=
 (\underbrace{\dots}_{m};\{p_1,q\};\underbrace{\dots}_{k_2-m};\underbrace{\dots}_{m};\underbrace{\dots}_{n-k_1-k_2}).\end{equation}

\textbf{Step 2}. For $v=(v_1;\underbrace{\dots}_{k_1-m};\underbrace{\dots}_{k_2-m}
v_1';
\underbrace{\dots}_{x-k_1-k_2})\in \mathbb{R}^x$ ($x=k$ or $n$) with $v_1 \in \mathbb{R}^{m}$ and $v_1'=(y_1,\dots,y_m)\in \mathbb{R}^{m}$, define 
\[\mathcal{P}_1 : \mathbb{R}^x\to \mathbb{R}^m,\ \mathcal{P}_1(v)=v_1; \quad \mathcal{P}_1': \mathbb{R}^x\to \mathbb{R}^m,\ \mathcal{P}_1'(v)=(y_m,\dots,y_1).\] 
Be careful of the definition of $\mathcal{P}_1'$, which sends $v$ to the reverse order of $v_1'$.

Notice that \eqref{1wl-r} implies 
\begin{equation}\label{w+<}
\mathcal{P}_1\big(w(\Lambda-\rho_c)+\rho_n-\zeta\big)\preceq   p_1^m+\mathcal{P}_1(\rho_n-\zeta)
\end{equation}
and the equality happens if and only if
\begin{equation}\label{w+=}w(\Lambda-\rho_c)=(p_1^{m};\{p_1,q\};
\underbrace{\dots}_{n-k_1}).\end{equation}

Moreover, 
\begin{equation}\label{w->}
\mathcal{P}_1'(w(\Lambda-\rho_c)+\rho_n-\zeta)\preceq p_1^m+\mathcal{P}_1'(\rho_n-\zeta)\end{equation}
and the equality happens if and only if
\begin{equation}\label{w-=}w(\Lambda-\rho_c)=(\underbrace{\dots}_{m},\{p_1,q\};\underbrace{\dots}_{k_2-m}; p_1^{m};\underbrace{\dots}_{n-k_1-k_2}).\end{equation}

By a direct computation, the right side of \eqref{w+<} and \eqref{w->} are given by
\[ p_1^m+\mathcal{P}_1(\rho_n-\zeta)=\{m,1\}  \quad \text{and}\quad p_1^m+\mathcal{P}_1'(\rho_n-\zeta)=\{-m,-1\}. \] 
Hence, the right side of \eqref{w+<} is equal to $(-1)$ multiple of that of \eqref{w->}, and $\mathcal{P}_1(\tau-\zeta)\preceq -\mathcal{P}_1'(\tau-\zeta)$. 
We will use Lemma \ref{kt_of_zs} to show that $\mathcal{P}_1(\tau-\zeta)\succeq -\mathcal{P}_1'(\tau-\zeta)$, so the equality happens, and both \eqref{w+<} and \eqref{w->} must be achieved, and so \eqref{w+=} and \eqref{w-=}, which imply \eqref{1wl-r}.

Let $\gamma$ be the component of $\tau=w(\Lambda-\rho_c)+\rho_n$ in \eqref{C=gam+h}.   

Apply (ii) of Lemma \ref{kt_of_zs}:
Notice that \eqref{w->} implies that $\mathcal{P}_1'(\tau-\zeta)\prec 0^{m}$. Combining with \eqref{C=gam+h}, one has $\mathcal{P}_1'(\gamma-\zeta_V)\prec 0^m$, which implies the conditions in (ii) of Lemma \ref{kt_of_zs} for $s=m$.

By (ii) of Lemma \ref{kt_of_zs}, we get that $\mathcal{P}_1(\gamma-\zeta_V)\succeq -\mathcal{P}_1'(\gamma-\zeta_V) \succ 0^m$.  Combining with $\mathcal{P}_1'(\tau-\zeta)\prec 0^m$ and \eqref{C=gam+h}, one get $\mathcal{P}_1(\tau-\zeta)\succeq -\mathcal{P}_1'(\tau-\zeta) \succ 0^m$. Therefore, $\mathcal{P}_1(\tau-\zeta)= -\mathcal{P}_1'(\tau-\zeta)$, and so both \eqref{w+<} and \eqref{w->}.

So one knows that any $\tau \in \mathcal{H}\cap \Omega$ is of the form
\begin{equation}\label{vstep1cd}
(p_1^{m};\{p_1,q\};\underbrace{\dots}_{k_2-m}; p_1^{m};\underbrace{\dots}_{n-k_1-k_2})+\rho_n.\end{equation}

Now one can use induction on $t$ to prove the case (III): 
For any $v=(v_1;v_0;v_2; v_1';v_3)\in \mathbb{R}^x$, where $x=k$ or $n$, $v_1,v_1'\in \mathbb{R}^m$, and \[v_0\in \mathbb{R}^{q-p_1+1},  v_2\in \mathbb{R}^{k_2-m}, v_3\in \mathbb{R}^{x-k_1-k_2}, \]
define $\Psi: \mathbb{R}^x\to \mathbb{R}^{x-k_1-m}$ by $\Psi(v)=(v_2;v_3)- \big(m^{k_2-m};0^{x-k_1-k_2}\big)$.

By definition of $\Omega$, it is not hard to check that: for any $v\in \mathbb{Z}^n$ of the form \eqref{vstep1cd}, $v\in \Omega$ only if $\Psi(v)$ is in $\Omega'$, where $\Omega'$ is the ``$\Omega$'' defined using  $\Lambda'=\mathcal{A}'\cup(\mathop{\cup}\limits_{i\geq 3}\mathcal{A}_i\cup\mathcal{U})$ instead of $\Lambda$, where $\mathcal{A}'=[\beta_1-k_1+1,\beta_2-k_2+2]$.

By \eqref{vstep1cd}, any $\tau \in \mathcal{H}\cap \Omega$ is of the form
\begin{equation}\label{v=zeta+d}
v=\zeta+(\{m,1\};0^{q-p_1+1};\underbrace{\dots}_{k_2-m};\{-1,-m\} ;\underbrace{\dots}_{n-k_1-k_2}).\end{equation}
Moreover, by the discussion before \eqref{vstep1cd}, the component $\gamma$ of $v$ in \eqref{C=gam+h} is of the form
\begin{equation}\label{gam1}
\gamma=\zeta+(\{m,1\};0^{q-p_1+1};\underbrace{\dots}_{k_2-m};\{-1,-m\} ;\underbrace{\dots}_{k-k_1-k_2}).\end{equation}

Combining with \eqref{mul_prod}, the multiplicity of $E_{\tau}^{\mathrm{U}(k)\times \mathrm{Sp}(r)}$ in $(V\boxtimes \mathbf{1})\otimes \mathcal{E}_1$ is 

\[ \sum_{E_{\gamma}^{\mathrm{U}(k)}\ \text{satisfies \eqref{gam1}}} \big[E_{\gamma}^{\mathrm{U}(k)}:V|_{\mathrm{U}(k)}\big]\cdot \big[E_{\gamma}^{\mathrm{U}(k)}\otimes E_{\mathcal{Q}_2(\tau)}^{\mathrm{U}(k)} : E_{\mathcal{Q}_1(\tau)}^{\mathrm{U}(k)} \big]  \]

Now by (iii) of Lemma \ref{kt_of_zs} for $\epsilon=m$, one has
\[\big[E_{\gamma}^{\mathrm{U}(k)}:V|_{\mathrm{U}(k)}\big]=\big[E_{\Psi(\gamma)}^{\mathrm{U}(k-k_1-m)}:V'|_{\mathrm{U}(k-k_1-m)}\big]\]
where $V'$ is the ``$V$'' defined by $\zeta'=\big((\beta_2-m)^{k_2-m};\beta_3^{k_3};\dots;\beta_t^{k_t}\big)$ instead of $\zeta_V$. Notice that $\zeta'$ corresponds to the $\Lambda'$ defined above.

Moreover, combining with \eqref{v=zeta+d} and \eqref{gam1}, the properties of $h_1,\dots,h_k$ imply that $\mathcal{Q}_2(\tau)$ contains more than $k_1+m$ zeros. One can use (ii) of Lemma \ref{red_tsp} to see 
\[\big[E_{\gamma}^{\mathrm{U}(k)}\otimes E_{\mathcal{Q}_2(\tau)}^{\mathrm{U}(k)} : E_{\mathcal{Q}_1(\tau)}^{\mathrm{U}(k)} \big]=\big[E_{\Psi(\gamma)}^{\mathrm{U}(k-k_1-m)}\otimes E_{\mathcal{Q}_2'\big(\Psi(\tau)\big)}^{\mathrm{U}(k-k_1-m)} : E_{\mathcal{Q}_1'\big(\Psi(\tau)\big)}^{\mathrm{U}(k-k_1-m)} \big]\]
where $\mathcal{Q}_1', \mathcal{Q}_2'$ is the ``$\mathcal{Q}_1, \mathcal{Q}_2$" for $(k-k_1-m,r)$ instead of $(k,r)$.

Hence, we get 
\begin{equation}\label{com_tt'}\big[E_{\tau}^{\mathrm{U}(k)\times \mathrm{Sp}(r)}: (V\boxtimes \mathbf{1})\otimes \mathcal{E}_1\big]= \big[E_{\tau'}^{\mathrm{U}(k-k_1-m)\times \mathrm{Sp}(r)}: (V'\boxtimes \mathbf{1})\otimes \mathcal{E}_1'\big],\end{equation}
where $\tau'=\Psi(\tau)$, and $\mathcal{E}_1'$ is the ``$\mathcal{E}_1$'' for $(k-k_1-m,r)$ instead of $(k,r)$. Let $\mathcal{H}'$ be the corresponding highest weights of $\mathrm{U}(k-k_1-m)\times \mathrm{Sp}(r)$-types in $(V'\boxtimes \mathbf{1})\otimes \mathcal{E}_1'$.

By induction on $t$, $\mathcal{H}'\cap \Omega'$ contains unique element, and the multiplicity of such $\mathrm{U}(k-k_1-m)\times \mathrm{Sp}(r)$-type shows up in $(V'\boxtimes \mathbf{1})\otimes \mathcal{E}_1'$ with multiplicity one. By \eqref{com_tt'}, one sees that $\mathcal{H}\cap \Omega$ contains at most one element. If such element exists, then such $\mathrm{U}(k)\times \mathrm{Sp}(r)$-type shows up in $(V\boxtimes \mathbf{1})\otimes \mathcal{E}_1$ with multiplicity one.

On the other hand, from an element $\tau'\in \mathcal{H}'\cap \Omega'$, similarly to the proof of (iii) of Lemma \ref{kt_of_zs}, one can get an element $\tau$ by ``inverting" the map $\Phi$, and such $\tau$ is in $\mathcal{H}\cap \Omega$.  
This finishes the proof of the sub-case (a) of Case (III), and also the proof of the proposition. 
\end{proof}
  
\begin{remark}
As the proof of the proposition above, the spin lowest $K$-type can be obtained by \eqref{i_tau} and \eqref{v=zeta+d}, which are exactly the algorithm \ref{algorithm1} and algorithm \ref{algorithm2}. This completes the proof of the Theorem \ref{thm-main}.
\end{remark}

\subsection{The uniqueness and multiplicity-free of spin lowest $K$-type in the $\mathrm{GL}(n,\mathbb{R})$ case}

By the similar argument as $\mathrm{GL}(n,\mathbb{H})$ case, we will prove that in the $\mathrm{GL}(n,\mathbb{R})$ case, the spin lowest $K$-type of any Dirac series is unique and multiplicity-free. This verifies the Conjecture 4.5 in \cite{Dong_Wong_2:2022} about uniqueness of spin lowest $K$-type for Dirac series of $\mathrm{GL}(n,\mathbb{R})$. 

\subsubsection{Dirac series of $\mathrm{GL}(n,\mathbb{R})$}
Let $n=2m$ (even) or $2m+1$ (odd), and $G=\mathrm{GL}(n,\mathbb{R})$, $K=\mathrm{O}(n)$, with their Lie algebra $\mathfrak{g}_0=\mathfrak{gl}(n,\mathbb{R})$, $\mathfrak{k}_0=\mathfrak{s}\mathfrak{o}(n)$.  Let $\mathfrak{t}_0\subset \mathfrak{k}_0$ be a maximal abelian subalgebra and let $\mathfrak{k}$, $\mathfrak{t}$ be the complexification of $\mathfrak{k}_0$, $\mathfrak{t}_0$.

From Section 4 of \cite{Dong_Wong_2:2022}, the Dirac series of $\mathrm{GL}(n,\mathbb{R})$ are consists of all the representations
\[\pi=\mathrm{Ind}_{\prod_{i=1}^t\mathrm{GL}(2k_i,\mathbb{R})\times \mathrm{GL}(r,\mathbb{R})}^{\mathrm{GL}(n,\mathbb{R})}\big(\mathop{\boxtimes}\limits_{i=1}^t S_{k_i,m_i+1}\boxtimes \pi_u\big),\]
where $\sum_{i=1}^t2k_i+r=n$, and
\begin{itemize}
\item[(i)] $S_{k_i,m_i+1}$ are Speh representations whose infinitesimal characters restricted to $\mathfrak{t}$ are
\[\mathcal{A}_i:=(m_i+2k_i-1,\dots,m_i+1)\in \mathfrak{t}^*,\]
and lowest $K$-types are $(m_i+k_i+1,\dots,m_i+k_i+1)$.
\item[(ii)] $\pi_u$ is the unipotent representation with non-zero Dirac cohomology with infinitesimal character $\mathcal{U}$. \item[(iii)] The infinitesimal character $\Lambda=\mathop{\bigcup}\limits_i^t \mathcal{A}_i\bigcup \mathcal{U}$ of $\pi$ 
is regular.
\end{itemize}

\subsubsection{Spin lowest $K$-types of Dirac series}
Following \cite[Section 5]{Vogan:1986}, each $K$-type is parametrized by a tuple $(\tau;\eta)$, where $\tau=(a_1,\dots,a_m)$ satisfies $a_1\geq \dots\geq a_m\geq 0$, and $\eta=0,1$ or $\frac{1}{2}$. 

Explicitly,
the highest weights in $E^K_{(\tau;\eta)}$ consist of $(a_1,\dots,\pm a_m)$ when $n=2m$, $a_m>0$ (in this case $\eta=\frac{1}{2}$); otherwise, the highest weights in $E^K_{(\tau;\eta)}$ consist of just $(a_1,\dots,a_m)$ (in this case, $\eta=0$ (or $1$) when the action of $\mathrm{diag}(1,\dots,1,-1)\in O(n)$ is trivial (resp.  non-trivial)). Let $\overline{\bullet}$ denote sign-flipping the last coordinate of $\bullet$.

Let \[\rho_n=\begin{cases} (m,m-1,\dots,1),& n=2m, \\ (m+1/2,m-1/2,\dots,3/2),& n=2m+1,\end{cases}\]
and
\[\rho_c=\begin{cases} (m-1,m-2,\dots,0),& n=2m, \\ (m-1/2,m-3/2,\dots,1/2),& n=2m+1.\end{cases}\]

From Section 1 of \cite{Dong_Wong_2:2022},  we know that a $K$-type $E^K_{(\tau;\eta)}$ of $\pi$ is the spin lowest $K$-type if and only if there exists $\sigma\in \{\rho_n, \overline{\rho_n}\}$ (for $n=2m$), or $\sigma=\rho_n$ (for $n=2m+1$), 
such that 
\[\{\tau-\sigma\}+\rho_c=\Lambda, \ \text{or} \ \{\overline{\tau}-\sigma\}+\rho_c=\Lambda.\]

A spin lowest $K$-type of $\pi$ can be found out by the algorithms in Section 4 of \cite{Dong_Wong_2:2022}, and its highest weight is denoted by $(\mathcal{S}_1,\dots,\mathcal{S}_t,\mathcal{S}_u)$. We will show that it is the unique spin lowest $K$-type of $\pi$ with multiplicity one by showing Theorem \ref{uni_m1_GLnR}.

Continuing the above, let $k=\sum_{i=1}^t k_i$, and set
\[L_1=\mathrm{GL}(k,\mathbb{C}), \ L_2=\mathrm{GL}(r,\mathbb{R});\ \ L=L_1\times L_2.\]
Let $\mathfrak{l}$ be the complexification of the Lie algebra of $L$. Let $\mathfrak{q}=\mathfrak{l}\oplus\mathfrak{u}$ be a $\theta$-stable parabolic subalgebra. Let $K_1=L_1\cap K\cong \mathrm{U}(k)$, and let $K_2=L_2\cap K\cong \mathrm{O}(r)$. As the notations in Section 4 of \cite{Dong_Wong_2:2022}, let $Z^{\#}$ be the Dirac series of $\mathrm{GL}(k,\mathbb{C})$ corresponding to the chains $\bigcup_{i=1}^t\mathcal{A}_i^{1}$, which is the chain obtained by adding all the coordinates of $\mathcal{A}_i$ by $1$. Then $Z^{\#}$ has lowest $K_1$-type with highest weight
\[\zeta_{Z^{\#}}:= \big((m_1+k_1+1)^{k_1};\dots;(m_t+k_t+1)^{k_t}\big).\]

Hence, as Section 4 of \cite{Dong_Wong_2:2022}, $\pi$ can be written as cohomological induced module $\mathcal{R}_{\mathfrak{q}}^{S}(Z\boxtimes \pi_u)$.  The  Blattner formula tells that for any $K$-type $E^K_{(\tau;\eta)}$, we have
\[ [\pi,E^K_{(\tau;\eta)}]_K=\sum_{i}(-1)^i\sum_{k=0}^{\infty}[(Z^{\#}\boxtimes \pi_u)\otimes S^k(\mathfrak{u}\cap\mathfrak{p}):
H^i(\mathfrak{u}\cap \mathfrak{k},E^K_{(\tau;\eta)})]_{L\cap K}.\]
By Blattner formula, to show that $\pi$ contains a unique and multiplicity-free spin lowest $K$-type, it suffices to show Theorem \ref{uni_m1_GLnR} (see the content above the theorem). 

Let's recall all special unipotent representations $\pi_u$ with non-zero Dirac cohomology, and the $K_2$-types of $\pi_u$, see Section 2 of \cite{Dong_Wong_2:2022}. Such $\pi_u$ can be the trivial representation or $\mathrm{sgn}(\det)$, or as follows.

\textbf{Case:} $r$ is odd,
\[\pi_u=\mathrm{Ind}^{\mathrm{GL}(r)}_{\mathrm{GL}(r-2b+1) \times \mathrm{GL}(2b-1)}(\mathrm{sgn}(\det)^{\epsilon_1}\otimes \mathrm{sgn}(\det)^{\epsilon_2}),\]
where $1\leq b\leq [\frac{r}{2}]+1, \epsilon_1, \epsilon_2 \in \{0,1\}.$
Its $K_2$-types are
\begin{equation}\label{kt_of_pi_u_od}
\pi_u|_{K_2}=\left\{\begin{array}{ll}
\mathop{\bigoplus}\limits_{a_1\geq \dots\geq a_q\geq 0}E^{K_2}_{(2a_1,\dots,2a_q,0,\dots,0;\epsilon')}, & \epsilon_1+\epsilon_2\ \text{is even},\\
\mathop{\bigoplus}\limits_{a_1\geq \dots\geq a_q\geq 0}E^{K_2}_{(2a_1+1,\dots,2a_q+1,0,\dots,0;\epsilon')}, & \epsilon_1+\epsilon_2\ \text{is odd},\end{array}\right.
\end{equation}
where $q=\min\{r-2b+1,2b-1\}$, $\epsilon'=\begin{cases}\epsilon_1, & q \ \text{is odd}\\ \epsilon_2, & q \ \text{is even}\end{cases}$.

\textbf{Case:} $r$ is even,  
\[\pi_u=\mathrm{Ind}^{\mathrm{GL}(r)}_{\mathrm{GL}(r-2b) \times \mathrm{GL}(2b-1)\times \mathrm{GL}(1)}(\mathrm{sgn}(\det)^{\epsilon_1}\otimes \mathrm{sgn}(\det)^{\epsilon_2}\otimes \mathrm{sgn}(\det)^{\epsilon_3}),\]
where $0< b\leq \frac{r}{2}$, $\epsilon_1, \epsilon_2, \epsilon_3 \in \{0,1\}$ such that $\epsilon_2+\epsilon_3\equiv b \mod 2$. Let $q=\min \{r-2b,2b-1\}$. 

When $0< b<\frac{r}{2}$, its $K_2$-types are
\[ 
\pi_u|_{K_2}=
\mathop{\bigoplus}\limits_{\tiny\begin{array}{c}a_1\geq \beta_1\geq a_2\geq \beta_2\geq  \dots\geq a_{q+1}\geq 0\\
\beta_1\geq \beta_2\geq\dots\geq \beta_q, \ \beta_i\equiv \epsilon_1+\epsilon_2\mod 2 \end{array}}E^{K_2}_{(a_1,\dots,a_{q+1},0,\dots,0;\epsilon')},\]
where $\epsilon'=\begin{cases} \epsilon_1, & q\ \text{is odd},\\ \epsilon_2, & q\ \text{is even}.\end{cases}$

When $b=\frac{r}{2}$, its $K_2$-types are
\[
\pi_u|_{K_2}=\left\{\begin{array}{ll}
\mathop{\bigoplus}\limits_{a\geq 0}E^{K_2}_{(2a,0,\dots,0;\epsilon_2)}, & \epsilon_2+\epsilon_3\ \text{is even},\\
\mathop{\bigoplus}\limits_{a \geq 0}E^{K_2}_{(2a+1,0,\dots,0;\epsilon_2)}, & \epsilon_2+\epsilon_3\ \text{is odd}.\end{array}\right.\]

Hence, all highest weights of the $K_2$-types of the $\pi_u$ above contain at least $z$ zeros, where  
\begin{equation}\label{z_eq}
z=\begin{cases}[\frac{r}{2}]-q, & n \ \text{is odd} \\ \frac{r}{2}-q-1, &  n \ \text{is even}\end{cases}.\end{equation}
Under the notation as above, the infinitesimal character of $\pi_u$ is
\[\mathcal{U}=\begin{cases}(r-q-1,\dots,q+2,q,q-1,\dots,1), & r\ \text{is odd},\\
(r-q-2,\dots,q+2,q,q-1,\dots,1,0), & r\ \text{is even}.\end{cases}\]

The following lemma is about the $K_1\times K_2\cong \mathrm{U}(k)\times \mathrm{O}(r)$-types in $S(\mathfrak{u}\cap\mathfrak{p})$.
\begin{lemma}\label{sup_glnr}
As $\mathrm{U}(k)\times \mathrm{O}(r)$-representation,
let 
\[\mathcal{E}_1=\bigoplus_{\tiny\begin{array}{c}a_1\geq \dots\geq a_x> 0, x\leq \min\{k,[\frac{r}{2}]\}\\ \epsilon\ \text{is determined as below}\end{array}}\big(E^{\mathrm{U}(k)}_{(a_1,\dots,a_x,0^{k-x})}\boxtimes E^{\mathrm{O}(r)}_{(a_1,\dots,a_x,0^{[\frac{r}{2}]-x};\epsilon)}\big),\]
where $\epsilon \equiv \sum\limits_{i=1}^x a_x \mod 2$ when $r$ is odd, and $\epsilon=\begin{cases}0, & x<[\frac{r}{2}]\\
\frac{1}{2}, &  x=[\frac{r}{2}]\end{cases}$ when $r$ is even.

Let 
\[\mathcal{E}_2=\bigoplus_{\tiny\begin{array}{c}a_1\geq \dots\geq a_k\geq 0;\ \ c_1\geq \dots \geq c_{[\frac{r}{2}]} \geq 0;\\ \sum\limits_i a_i>\sum\limits_j c_j; \epsilon'\in \{0,1,\frac{1}{2}\}\end{array}}
\big(E^{\mathrm{U}(k)}_{(a_1,\dots,a_k)}\boxtimes E^{\mathrm{O}(r)}_{(c_1,\dots,c_{[\frac{r}{2}]};\epsilon')}\big).\]
Then $\mathcal{E}_1\oplus \mathcal{E}_2$ contain the set of $K_1\times K_2\cong \mathrm{U}(k)\times \mathrm{O}(r)$-types showing up in $S(\mathfrak{u}\cap\mathfrak{p})$, where the sets are not counting the multiplicities. Moreover, every $\mathrm{U}(k)\times \mathrm{O}(r)$-types in $\mathcal{E}_1$ shows up in $S(\mathfrak{u}\cap\mathfrak{p})$ with multiplicity one.
\end{lemma}

\begin{proof}
Firstly, as $\mathrm{U}(k)\times \mathrm{O}(r)$-representation,
\[\mathfrak{u}\cap\mathfrak{p} \cong \big(E^{\mathrm{U}(k)}_{(1,0,\dots,0)}\boxtimes E^{\mathrm{O}(r)}_{(1,0,\dots,0;\epsilon_1)}\big)\oplus \big(E^{\mathrm{U}(k)}_{(2,0,\dots,0)}\boxtimes E^{\mathrm{O}(r)}_{(0,0,\dots,0;\epsilon_0)}\big),\]
where $\epsilon_0=0$, and $\epsilon_1=\begin{cases}1, & r\ \text{is odd}, \\ 0, & r\ \text{is even},\end{cases}$ except $\epsilon_1=\frac{1}{2}$ when $r=2$.

Notice that $S\big(E^{\mathrm{U}(k)}_{(1,0,\dots,0)}\boxtimes E^{\mathrm{O}(r)}_{(1,0,\dots,0;\epsilon_1)}\big)$ is the restriction of $\mathrm{U}(k)\times \mathrm{U}(r)$-module \[S\big(E^{\mathrm{U}(k)}_{(1,0,\dots,0)}\boxtimes E^{\mathrm{U}(r)}_{(1,0,\dots,0)}\big)\]
to $\mathrm{U}(k)\times \mathrm{O}(r)$. By Corollary 5.6.6 of \cite{Goodman_Wallach:2009},
\[S\big(E^{\mathrm{U}(k)}_{(1,0,\dots,0)}\boxtimes E^{\mathrm{U}(r)}_{(1,0,\dots,0)}\big)\cong \bigoplus_{a_1\geq \dots\geq a_x\geq 0, x\leq \min\{k,r\}}\big(E^{\mathrm{U}(k)}_{(a_1,\dots,a_x,0^{k-x})}\boxtimes E^{\mathrm{U}(r)}_{(a_1,\dots,a_x,0^{r-x})}\big).\]
We set $a_i=0$ for $i>x$. By a similar argument as Lemma \ref{GLtoSp}, the highest weights of $\mathrm{O}(r)$-types in $E^{\mathrm{U}(r)}_{(a_1,\dots,a_x,0^{r-x})}$ are of the form $(a_1-a_{r},a_2-a_{r-1},\dots,a_{[\frac{r}{2}]}-a_{[\frac{r}{2}]+1})+\delta$ with 
\[\delta=\left\{\begin{array}{ll}\sum\limits_{1\leq i<j\leq [\frac{r}{2}],\ a_{ij}\leq 0} a_{ij}(e_i-e_j)+\sum\limits_{1\leq i<j\leq [\frac{r}{2}],\ c'_{ij}\leq 0}c'_{ij}(e_i+e_j),& r\ \text{is even},\\ \sum\limits_{1\leq i<j\leq [\frac{r}{2}],\ a_{ij}\leq 0} a_{ij}(e_i-e_j)+\sum\limits_{1\leq i<j\leq [\frac{r}{2}],\ c'_{ij}\leq 0} c'_{ij}(e_i+e_j)+\sum\limits_{1\leq i\leq [\frac{r}{2}],\ c'_i\leq 0}c'_ie_i, & r\ \text{is odd}.\end{array}\right.\]

If $a_{[\frac{r}{2}]+1}>0$, then the $\mathrm{U}(k)\times \mathrm{O}(r)$-types in $E^{\mathrm{U}(k)}_{(a_1,\dots,a_x,0,\dots,0)}\boxtimes E^{\mathrm{U}(r)}_{(a_1,\dots,a_x,0,\dots,0)}$ are contained in $\mathcal{E}_2$. If $a_{[\frac{r}{2}]+1}=0$, by the Littlewood-Richardson branching law, see \cite{Knapp:2002}, Theorem 9.75, we know that the $\mathrm{O}(r)$-types in $E^{\mathrm{U}(r)}_{(a_1,\dots,a_r)}$ consist of $E^{\mathrm{O}(r)}_{(a_1,\dots,a_{[\frac{r}{2}]};\epsilon)}$ (with $\epsilon$ defined the same as $\mathcal{E}_1$) with multiplicity one and some $E^{\mathrm{O}(r)}_{(c_1,\dots,c_{[\frac{r}{2}]};\epsilon')}$
with $\sum\limits_i c_i<\sum\limits_i a_i$.

So the set of $\mathrm{U}(k)\times \mathrm{O}(r)$-types in $S\big(E^{\mathrm{U}(k)}_{(1,0,\dots,0)}\boxtimes E^{\mathrm{O}(r)}_{(1,0,\dots,0;\epsilon_1)}\big)$ are contained in $\mathcal{E}_1\oplus \mathcal{E}_2$ and all $\mathrm{U}(k)\times \mathrm{O}(r)$-types of $\mathcal{E}_1$ show up in $S\big(E^{\mathrm{U}(k)}_{(1,0,\dots,0)}\boxtimes E^{\mathrm{O}(r)}_{(1,0,\dots,0;\epsilon_1)}\big)$ with multiplicity one.

Moreover, when $l>0$, for any $m\in \mathbb{N}$, the set of $\mathrm{U}(k)\times \mathrm{O}(r)$-types in
\[S^m\big(E^{\mathrm{U}(k)}_{(1,0,\dots,0)}\boxtimes E^{\mathrm{U}(r)}_{(1,0,\dots,0)}\big)\otimes S^l\big(E^{\mathrm{U}(k)}_{(2,0,\dots,0)}\boxtimes E^{\mathrm{U}(r)}_{(0,0,\dots,0)}\big)\]
are contained in $\mathcal{E}_2$. Now the lemma follows.
\end{proof} 

\begin{lemma}\label{tensor_o}
Let $1\leq p,q\leq [\frac{r}{2}]$ such that $p+q\leq [\frac{r}{2}]$. Let $a=(a_1,\dots,a_p,0^{[\frac{r}{2}]-p})$ and $b=(b_1,\dots,b_q,0^{[\frac{r}{2}]-q})$ be two highest weight of $\mathrm{SO}(n)$. Let $\alpha_1\geq \dots \geq \alpha_p\geq \beta_1\geq \cdots\geq \beta_q\geq 0$ be the integers such that $\sum\limits_{j=1}^q\beta_j=\sum\limits_{j=1}^qb_j$. Then 
\[\big[E_{(a;\eta_1)}^{\mathrm{O}(r)}\otimes E_{(b;\eta_2)}^{\mathrm{O}(r)}: E_{(\alpha_1,\dots,\alpha_p,\beta_1,\dots,\beta_q,0^{[\frac{r}{2}]-p-q};\eta)}^{\mathrm{O}(r)}\big]>0\]
only if $\alpha_i=a_i$ for $1\leq i\leq p$, $\beta_j=b_j$ for $1\leq j\leq q$, and $\eta$ is uniquely determine by $\eta_1$ and $\eta_2$. When they happen, the multiplicity is exactly one.
\end{lemma}
\begin{proof} 
Let $\xi=(\alpha_1,\dots,\alpha_p,\beta_1,\dots,\beta_q,0^{[\frac{r}{2}]-p-q})$. By decomposition of tensor products of $\mathrm{O}(r)$, see Theorem 2.1.2 of \cite{Howe_Tan_Willenbring:2005}, one has 
\[\big[E_{a}^{\mathrm{O}(r)}\otimes E_{b}^{\mathrm{O}(r)}: E_{\xi}^{\mathrm{O}(r)}\big]=\sum_{\lambda,\mu,\nu}c^{\xi}_{\lambda,\mu}c^{a}_{\lambda,\nu}c^{b}_{\mu,\nu},\]
 where $\lambda,\mu,\nu$ run over all partitions with non-negative integer entries, and $c^{u}_{v,w}$ is the Littlewood-Richardson coefficient, that is, $[E_{v}^{\mathrm{U}(r)}\otimes E_{w}^{\mathrm{U}(r)}:E_{u}^{\mathrm{U}(r)}]_{\mathrm{U}(r)}$.

Assume that there are some non-negative integer partitions $\lambda,\mu,\nu$ such that $c^{\xi}_{\lambda,\mu}c^{a}_{\lambda,\nu}c^{b}_{\mu,\nu}\neq 0$. Since $a$ has at most $p$ non-zero entries, if $c^{a}_{\lambda,\nu}\neq 0$, then $\lambda$ also has at most $p$ non-zero entries. Since $c^{\xi}_{\lambda,\mu}\neq 0$, one has $\sum\limits_{j=1}^q \beta_j\leq \sum\limits \mu$, where $\sum\limits \mu$ means summing up all coordinates of $\mu$. Moreover, since $c^{b}_{\mu,\nu}\neq 0$, one has $\sum\limits_{j=1}^q b_j= \sum \mu+\sum \nu\geq \sum \mu$. Combining with the condition $\sum\limits_{j=1}^q b_j=\sum\limits_{j=1}^q \beta_j$, we know $\sum \nu=0$, i.e. $\nu=0$. By $c^{a}_{\lambda,\nu}c^{b}_{\mu,\nu}\neq 0$, one gets $\lambda=a$ and $\mu=b$.

So $\big[E_{a}^{\mathrm{O}(r)}\otimes E_{b}^{\mathrm{O}(r)}: E_{\xi}^{\mathrm{O}(r)}\big]_{\mathrm{O}(r)}=c^{\xi}_{a,b}c^{a}_{a,0}c^{b}_{b,0}$. When $c^{\xi}_{a,b}\neq 0$, since $\sum\limits_{j=1}^qb_j=\sum\limits_{j=1}^q\beta_j$, one has  $\beta_j=b_j$ for $1\leq j\leq q$, and $\alpha_i=a_i$ for $1\leq i\leq p$. In this case, the multiplicity is one.
 
\end{proof}

Let us discuss the spin lowest $K$-type: as above, a $K$-type of $\pi$, $E_{(\tau;\eta)}^K$, is spin lowest $K$-type if and only if $\{\tau-\sigma\}+\rho_c=\Lambda$ or $\{\overline{\tau}-\sigma\}+\rho_c=\Lambda$, equivalently, (i) $\tau=w(\Lambda-\rho_c)+\sigma$ or (ii) $\overline{\tau}=w(\Lambda-\rho_c)+\sigma$. 
When $\eta\neq \frac{1}{2}$, one only need to consider the case (i). When $\eta=\frac{1}{2}$ (so $n$ is even, the last coordinate of $\tau$ is positive),  if $\overline{\tau}=w(\Lambda-\rho_c)+\sigma$, then  $\tau=\overline{w(\Lambda-\rho_c)}+\overline{\sigma}$. Notice that the last coordinate of $\rho_c$ and $\Lambda$ are zeros, so $w(\Lambda-\rho_c)$ contains some zero coordinate. Hence, $\overline{w(\Lambda-\rho_c)}$ is equal to $w'(\Lambda-\rho_c)$ for some $w'\in W(D_m)$. In conclusion, one only need to consider the case (i). Therefore, to show that $\pi$ contains a unique and multiplicity-free spin lowest $K$-type, it suffices to show Theorem \ref{uni_m1_GLnR}. 

\begin{theorem}\label{uni_m1_GLnR}
There is exactly one element $\tau$ in
\[
\Omega_0:=\left\{w(\Lambda-\rho_c)+\sigma\ \Big\vert\ \begin{array}{c} w\in W(\mathfrak{k},\mathfrak{t})\ \text{s.t.}\  w(\Lambda-\rho_c)+\sigma\ \text{is}\ \Delta^+(\mathfrak{k},\mathfrak{t})\text{-dominant}\\ \text{and has non-negative last coordinate}\end{array}\right\} \]
such that
\begin{equation}\label{cup>0_R}
\big[(Z^{\#}\boxtimes \pi_u)\otimes S^k(\mathfrak{u}\cap\mathfrak{p}):
H^i(\mathfrak{u}\cap \mathfrak{k},E^K_{(\tau;\eta)})\big]_{L\cap K}>0
\end{equation}
for some $i\in \mathbb{N}$ and some $\eta$. When \eqref{cup>0_R} happens, then $i=0$, $\eta$ is also uniquely determined, and the multiplicity \eqref{cup>0_R} is one.
\end{theorem}

\begin{proof}

\noindent {\bf Step 1:} We characterize  the sets of highest weights of $L\cap K$-types in $(Z^{\#}\boxtimes\pi_u)\otimes S^k(\mathfrak{u}\cap\mathfrak{p})$ and $\mathop{\oplus}\limits_{\tau\in \Omega_0} \mathop{\oplus}\limits_{i\in \mathbb{N}} H^i(\mathfrak{u}\cap \mathfrak{k},E^K_{(\tau;\eta)})$. 
 
For $(Z^{\#}\boxtimes\pi_u)\otimes S(\mathfrak{u}\cap\mathfrak{p})$:  
By Lemma \ref{sup_glnr},  the highest weights of the $L\cap K$-types in $(Z^{\#}\boxtimes \pi_u)\otimes S(\mathfrak{u}\cap\mathfrak{p})$ are contained in the set
\[\mathcal{A}:=\text{the set of highest weights of }\ L\cap K\text{-types in}  (Z^{\#}\boxtimes \pi_u)\otimes
(\mathcal{E}_1 \oplus
\mathcal{E}_2).\]

For $\mathop{\oplus}\limits_{\tau\in \Omega_0} \mathop{\oplus}\limits_{i\in \mathbb{N}} H^i(\mathfrak{u}\cap \mathfrak{k},E^K_{(\tau;\eta)})$:
we have for any $K$-type $E^K_{(\tau;\eta)}$,
\[H^i(\mathfrak{u}\cap \mathfrak{k},E^K_{(\tau;\eta)})=\bigoplus_{l(w_1)=i,w_1\in W^1}E^{L\cap K}_{(w_1(\tau+\rho_c)-\rho_c;\eta')}\oplus \bigoplus_{l(w_1)=i,w_1\in W^1}E^{L\cap K}_{(w_1(\overline{\tau}+\rho_c)-\rho_c;\eta'')},\]
where the second term (involving  $\overline{\tau}$) shows up only when $n$ is even, and the last coordinate of $\tau$ is positive (that is, $\eta=\frac{1}{2}$). Here,
\[W^1=\{w_1\in W(\mathfrak{k},\mathfrak{t})\ |\ \langle w_1\rho_c,\alpha\rangle>0, \forall \alpha\in \Delta^+(\mathfrak{l}\cap \mathfrak{k},\mathfrak{t})\},\]
and $\eta',\eta''\in \{0,1,\frac{1}{2}\}$ are uniquely determined by $\eta$ and $w_1$.  Since $E^{L\cap K}_{(w_1(\tau+\rho_c)-\rho_c;\eta)}$ is a $\mathrm{U}(k)\times \mathrm{O}(r)$-type, the last coordinate of $w_1(\tau+\rho_c)-\rho_c$ is non-negative (Similarly for $w_1(\overline{\tau}+\rho_c)-\rho_c$). 

 Assume that a $L\cap K$-type  $E^{L\cap K}_{(v;\widetilde{\eta})}$ shows up in $\mathop{\oplus}\limits_{\tau\in \Omega_0}\mathop{\oplus}\limits_{i\in \mathbb{N}}H^i(\mathfrak{u}\cap \mathfrak{k},E_{(\tau;\eta)}^K)$, then $v$ is contained in
\[
\begin{aligned}\mathcal{B}= &\ \{w_1(\tau+\rho_c)-\rho_c\ |\  \tau\in \Omega_0, w_1\in W^1,\ \text{the last coordinate of} \ w_1(\tau+\rho_c)-\rho_c \ \geq 0 \}  \\ \cup  &\ \{w_1(\overline{\tau}+\rho_c)-\rho_c\  |\ \tau\in \Omega_0,    w_1\in W^1,\ \text{the last coordinate of} \ w_1(\overline{\tau}+\rho_c)-\rho_c \ \geq 0\}.\end{aligned}\]

\noindent {\bf Step 2:} Further approximation of the set  $\mathcal{A}\cap \mathcal{B}$. 

\textbf{Claim:} Let $y\in \mathcal{A}\cap \mathcal{B}$, then $y$ must be a highest weight of  $(Z^{\#}\boxtimes \pi_u)\otimes
\mathcal{E}_1$ ( a subset of $\mathcal{A}$). Moreover, $y$ can not be of the form $w_1(\tau+\rho_c)-\rho_c, w_1\in (W^1\setminus \{1\})$ (or $w_1(\overline{\tau}+\rho_c)-\rho_c$, $w_1\in W^1$),  and $y=\tau$ with $\tau=w(\lambda-\rho_c)+\sigma$ an element in $\Omega_0$  such that $w(\Lambda-\rho_c)$ satisfies \eqref{w=1/2} when $n$ is odd or \eqref{w_rc_ev} when $n$ is even. 

Similar to the proof of Theorem \ref{uni_m1_GLnH}, the main idea is to show that $\phi(y)=\phi\big((\zeta_{Z^{\#}};0^{[\frac{r}{2}]})\big)$ with $\phi$ defined by
\[\phi:\mathbb{R}^m\to \mathbb{R}, \ (y_1,\dots,y_m)\mapsto \sum_{i=1}^{k}y_i-\sum_{j=m-z+1}^m y_j,\]
where $m=[\frac{n}{2}]=k+[\frac{r}{2}]$ and $z$ is defined as \eqref{z_eq}. Notice that when $z=0$, one has $\Lambda-\rho_c$ is either $(0,\dots,0)$ or $(1,\dots,1)$. In these cases, it is easier to prove the theorem by a similar way.  From now on, we assume $z\neq 0$. 

Similar to the Theorem \ref{uni_m1_GLnH}, by definition of $\mathcal{A}$, and definition of $\mathcal{E}_1$, $\mathcal{E}_2$ in Lemma \ref{sup_glnr}, one has
\begin{equation}\label{phi_A_R}
\phi(y)\geq \phi\big((\zeta_{Z^{\#}};0^r)\big).\end{equation}

For any $x\in \mathbb{N}_+$, define $S_0: \mathbb{R}^x\to \mathbb{R}$, by $S_0\big((y_1,\dots,y_x)\big)=\sum\limits_{i=1}^x y_i$. Define \[\mathcal{N}:=\{(x_1,\dots,x_k,y_1,\dots,y_{[\frac{r}{2}]})\in \mathbb{R}^{k+[\frac{r}{2}]}\ |\ x_{\bullet}\leq 0, y_{\bullet}\geq 0\}.\]
To show for $y\in \mathcal{B}$, \begin{equation}\label{phi_B_R}\phi(y)\leq\phi\big((\zeta_{Z^{\#}};0^r)\big),\end{equation}
it is more convenient to discuss the cases ``$n$ is odd" and ``$n$ is even" separably.

\textbf{Case}: $n=2m+1$ is odd.

For any $w_1\in W^1$,  one has
\[w_1\rho_c-\rho_c\in \mathcal{N},\ w_1\tau-\tau\in \mathcal{N};\] 
So for $
y=w_1(\tau+\rho_c)-\rho_c=\tau+(w_1\tau-\tau)+(w_1\rho_c-\rho_c)$,
\begin{equation}\label{w_rho_r}\phi(y)\leq \phi(\tau)=\phi\big(w(\Lambda-\rho_c)+\sigma)=\phi(w(\Lambda-\rho_c))+\phi(\sigma),\end{equation}
and $\phi(y)$ equals to the right side only if $(w_1\tau-\tau)+(w_1\rho_c-\rho_c)=(0^{k},\underbrace{\dots}_{[\frac{r}{2}]-z},0^z)$, which implies $w_1=1$, and $y=\tau$.

Notice that $\Lambda-\rho_c$ is a vector whose entries are positive half integers and contain at least $(r-1)/2-z$ ``$1/2$'s", so we have
\begin{equation}\label{w_1/2}\phi(w(\Lambda-\rho_c))\leq S_0(\Lambda-\rho_c)-((r-1)/2-z)/2,\end{equation}
and \eqref{w_1/2} becomes an equality if and only if
\begin{equation}\label{w=1/2}
w(\Lambda-\rho_c)=(\underbrace{\dots}_{k\ \text{entries}\  \geq 0},\underbrace{\dots}_{\frac{r-1}{2}-z\ \text{entries in}\ \{\pm\frac{1}{2}\}},\underbrace{\dots}_{z\ \text{entries}\   \leq 0}).
\end{equation} 

A direct calculation shows that
\[\phi(\sigma)-((r-1)/2-z)/2=S_0(\sigma)-S_0(\mathcal{U})-(r-1)/2.\]

Combining \eqref{w_1/2}, we have
\[\phi\big(w(\Lambda-\rho_c)\big)+\phi(\sigma)
\leq S_0(\Lambda-\rho_c)+S_0(\sigma)-S_0(\mathcal{U})-(r-1)/2.\]

Moreover, it is not hard to see that $S_0(\Lambda-\rho_c)+S_0(\sigma)-S_0(\mathcal{U})-(r-1)/2=\phi\big((\zeta_{Z^{\#}};0)\big)$.
Hence,
\begin{equation}\label{phi_b}
\phi\big(w(\Lambda-\rho_c)+\sigma\big)\leq \phi\big((\zeta_{Z^{\#}};0)\big).\end{equation}

Combining \eqref{w_rho_r} and \eqref{phi_b}, we get \eqref{phi_B_R} when $n$ is odd. By \eqref{phi_A_R}, for any $y\in \mathcal{A}\cap \mathcal{B}$, we have $\phi(y)=\phi\big((\zeta_{Z^{\#}};0)\big)$. By the discussion above, the claim follows when $n$ is odd.

\textbf{Case}: $n=2m$ is even.    

Let us show that
\begin{equation}\label{phi_B_R_ev}
\phi(y)\leq S_0(\Lambda-\rho_c)+\phi(\rho_n), \ \forall \ y=w_1(\tau+\rho_c)-\rho_c\ \text{or}\ w_1(\overline{\tau}+\rho_c)-\rho_c\in \mathcal{B},\end{equation}
where $w_1\in W^1$, and $\tau=w(\Lambda-\rho_c)+\sigma\in \Omega_0$.

Notice that 
\[\phi(w(\Lambda-\rho_c))\leq S_0(\Lambda-\rho_c)+\phi(\rho_n)\]
and equality holds if and only if
\begin{equation}\label{w_rc_ev}
w(\Lambda-\rho_c)=(\underbrace{\dots}_{k\ \text{entries}\  \geq 0},0^{\frac{r}{2}-z},\underbrace{\dots}_{z\ \text{entries}\   \leq 0}). \end{equation}

When $\sigma=\rho_n$, $y=w_1(\tau+\rho_c)-\rho_c$, one has:
$(w_1\tau-\tau)+(w_1\rho_c-\rho_c)\in \mathcal{N}$ by that $w_1\in W^1$ and the last coordinate of $y$ is non-negative. Hence, $\phi(y)\leq \phi(\tau)=\phi(w(\Lambda-\rho_c))+\phi(\sigma)\leq S_0(\Lambda-\rho_c)+\phi(\rho_n)$, and so the \eqref{phi_B_R_ev}, and the equality holds only when $w_1=1$ and  \eqref{w_rc_ev} happens.

When $\sigma=\overline{\rho_n}$, $y=w_1(\tau+\rho_c)-\rho_c$,  one has:
$(w_1\tau-\tau)+(w_1\rho_c-\rho_c)\in \mathcal{N}$. Assume the last coordinate of $\tau$ is $a$ ($>0$), then $\phi(y)\leq \phi(\tau)=\phi(w(\Lambda-\rho_c))+\phi(\sigma)\leq S_0(\Lambda-\rho_c)-2(a+1)+\phi(\sigma)<S_0(\Lambda-\rho_c)+\phi(\rho_n)$, 
where the third to the last inequality is due to the last coordinate of $w(\Lambda-\rho_c)$ is $a+1$, since the last coordinate of $\tau=w(\Lambda-\rho_c)+\sigma$ is $a$. 

When $\sigma=\rho_n$, $y=w_1(\overline{\tau}+\rho_c)-\rho_c$, assume the last coordinate of $\tau$ is $a$ ($>0$). Then $\phi(y)\leq\phi(\overline{\tau})-2a=\phi(\tau)\leq S_0(\Lambda-\rho_c)+\phi(\rho_n)-2(a-1)\leq S_0(\Lambda-\rho_c)+\phi(\rho_n)$.
Here, the first inequality is due to that the last coordinate of $y$ is non-negative, then the last coordinate of $y$ must be $\geq a$. The equality happens only when $a=1$, $w_1\tau=\overline{\tau}$, which is impossible (since $w_1$ must change even number of signs).

When $\sigma=\overline{\rho_n}$, $y=w_1(\overline{\tau}+\rho_c)-\rho_c$, by similar discussion, $\phi(y)<S_0(\Lambda-\rho_c)+\phi(\rho_n)$.

In conclusion,  $\phi(y)\leq S_0(\Lambda-\rho_c)+\phi(\rho_n)$, and equality happens only when $\sigma=\rho_n$, $w_1=1$, and \eqref{w_rc_ev} holds. It is not hard to see that $S_0(\Lambda-\rho_c)+\phi(\rho_n)=\phi\big((\zeta_{Z^{\#}};0)\big)$.
So the claim follows when $n$ is even.

\noindent {\bf Step 3:} Let $\mathcal{H}$ be the set of highest weights, with non-negative last coordinate, of $L\cap K$-types in $(Z^{\#}\boxtimes \mathbf{1})\otimes \mathcal{E}_1$, and let 
\[ \Omega= \big\{w(\Lambda-\rho_c)+\rho_n\ |\ w(\Lambda-\rho_c)+\rho_n\in \Omega_0,\ w(\Lambda-\rho_c)\ \text{satisfies \eqref{w=1/2}($n$ is odd) or \eqref{w_rc_ev}($n$ is even)}\big\}. \]

By step 2, to prove the theorem, it suffices to show 
\[\sum_{\eta\in\{0,1,\frac{1}{2}\}}\big[(Z^{\#}\boxtimes \pi_u)\otimes \mathcal{E}_1: \mathop{\oplus}\limits_{\tau\in \Omega}  E^{L\cap K}_{(\tau;\eta)}\big]_{L\cap K}=1.\]

We give the details for the proof of a basic case when $n=2m+1$ is odd. Other cases can be proved similarly as Proposition \ref{prop_ktype}.

Assume the chains are interlaced as follows:
\[\begin{array}{c} (\mathcal{A}_1)   \dots \  (\mathcal{A}_t)\\
 (\quad \quad \quad \mathcal{U} \quad \quad \quad) \end{array},\]
where $\mathcal{A}_i=(m_i+2k_i-1,\dots,m_i+1)$ and
$\mathcal{U}=(r-q-1,\dots,q+2,q,q-1,\dots,1)$ satisfy $r-q-1>m_1+2k_1-1\geq m_1+1>\dots>m_t+2k_t-1\geq m_t+1>q$.
Recall that $\rho_c=\{\frac{n-2}{2},\frac{1}{2}\}$, and $\rho_n=\{\frac{n}{2},\frac{3}{2}\}$.
One has 
\[\Lambda-\rho_c=
(\{s,p_1\};p_1^{2k_1-1};\{p_1,p_2\};p_2^{2k_2-1};\dots;p_t^{2k_t-1};\{p_t,\frac{1}{2}\};(\frac{1}{2})^{q-1}),\]
for $s=r-q-\frac{n}{2}$, $p_i=\frac{m_i-n+r-q+1}{2}+\sum\limits_{j=1}^i k_j$. Notice that $s-\frac{1}{2}+2\sum\limits_{i=1}^tk_i+q=[\frac{n}{2}]$. 

Assume that $w(\Lambda-\rho_c)+\rho_n\in \mathcal{H}\cap \Omega$. By the property of $\mathcal{H}$, we know the last $[\frac{r}{2}]-q$($=s-\frac{1}{2}+\sum\limits_{i=1}^tk_i$) entries of the elements in $\mathcal{H}$ contain at most $k$ non-zero entries.
Since $w(\Lambda-\rho_c)+\rho_n\in \mathcal{H}$, we have
$$
w(\Lambda-\rho_c)=(\underbrace{\dots}_{[\frac{n}{2}]-s+\frac{1}{2}};\{-s,-3/2\}).$$

Applying the similar argument to Case (I) in the Proposition \ref{prop_ktype} to the $(k-k_t+1)^{st}$ through $k^{th}$ entries and the $(k+q+1)^{th}$ to the $(k+q+k_t)^{th}$ entries of $w(\Lambda-\rho_c)+\rho_n$, one gets
\[w(\Lambda-\rho_c)=(\underbrace{\dots}_{k-k_t},p_t^{k_t}; \underbrace{\dots}_{q};(-p_t)^{k_t};\underbrace{\dots}_{k-k_t};\{-s,-3/2\}).\]
Similarly, we can determine the rest of entries of $w(\Lambda-\rho_c)$. Then one gets
\begin{equation}\label{w_uniq_R}w(\Lambda-\rho_c)=(p_1^{k_1};\dots;p_t^{k_t}; \underbrace{\cdots \star \cdots}_{q};(-p_t)^{k_t};\dots,(-p_1)^{k_1};\{-s,-3/2\}),\end{equation}
where $\underbrace{\cdots \star \cdots}_{q}$ are in $\{\pm \frac{1}{2}\}$.

For any $\delta\in \{\pm \frac{1}{2}\}^{q}$, let $y_{\delta}=w(\Lambda-\rho_c)+\rho_n$ satisfying \eqref{w_uniq_R} with $\underbrace{\cdots \star \cdots}_{q}=\delta$.
To prove the theorem, it suffices to show 
\[\sum_{\eta\in\{0,1\}}\sum_{\delta\in \{\pm \frac{1}{2}\}^{q}} \big[(Z^{\#}\boxtimes \pi_u)\otimes \mathcal{E}_1:  E^K_{(y_{\delta};\eta)}\big]_{L\cap K}=1.\]
Notice that
\begin{equation}\label{yd-z}
\begin{array}{rl} & y_{\delta}-\big(\zeta_{Z^{\#}};0^{[\frac{r}{2}]}\big) \vspace{0.8em} \\
= & \big(\{\frac{r-q-m_1-1}{2},\frac{r-q-m_1-2k_1+1}{2}\};\dots;\{\frac{r-q-m_t-1}{2},\frac{r-q-m_t-2k_t+1}{2}\}; \delta+\{\frac{n-2k}{2},\frac{n-2k-2q+2}{2}\}; 
 \vspace{0.8em}
\\ & 
\{\frac{r-q-m_t-1}{2},\frac{r-q-m_t-2k_t+1}{2}\};\dots; \{\frac{r-q-m_1-1}{2},\frac{r-q-m_1-2k_1+1}{2}\}; 0^{s-\frac{1}{2}}\big),\end{array}\end{equation}
Define $\mathcal{Q}_1: \mathbb{R}^{[\frac{n}{2}]} \to \mathbb{R}^k$ and $\mathcal{Q}_2: \mathbb{R}^{[\frac{n}{2}]} \to \mathbb{R}^{[\frac{r}{2}]}$ by 
\[\begin{aligned}
& \mathcal{Q}_1\big((x_1,\dots,x_k;y_1,\dots,y_{[\frac{r}{2}]})\big)=(x_1,\dots,x_k),\\ 
& \mathcal{Q}_2\big((x_1,\dots,x_k;y_1,\dots,y_{[\frac{r}{2}]})\big) = (y_1,\dots,y_{[\frac{r}{2}]}).\end{aligned}\] 
Then by the property of tensor product, one has
\begin{equation}\label{mul=sum}\begin{array}{rl}& \big[(Z^{\#}\boxtimes \pi_u)\otimes \mathcal{E}_1:  E^K_{(y_{\delta};\eta)}\big]_{L\cap K}\vspace{0.8em} \\  = & \sum\limits_{E_{\gamma}^{K_1}}\sum\limits_{E_{(\chi;\varepsilon)}^{K_2}}\sum\limits_{\tiny\begin{array}{c}a=(a_1,\dots,a_k)\in \mathbb{N}^{k}\\ a_1\geq\dots\geq a_k\geq 0\end{array}}\big[E_{\gamma}^{K_1}:Z^{\#}|_{K_1}\big]_{K_1}\cdot\big[E_{\gamma}^{K_1}\otimes E_a^{K_1}: E_{\mathcal{Q}_1(y_{\delta})}^{K_1}]_{K_1}\cdot \vspace{0.8em} \\& 
[E_{(\chi;\varepsilon)}^{K_2}:\pi_u|_{K_2}]_{K_2}\cdot [E_{(\chi;\varepsilon)}^{K_2}\otimes E_{\tiny \big(a,0^{[\frac{r}{2}]-k};\epsilon(a)\big)}^{K_2}: E_{(\mathcal{Q}_2(y_{\delta});\eta)}^{K_2}]_{K_2}, \end{array}\end{equation}
where $\epsilon(a)\equiv \sum\limits_{i=1}^ka_i$ (mod 2) as Lemma \ref{sup_glnr}.

Let
$a'=\mathcal{Q}_1\Big(y_{\delta}-\big(\zeta_{Z^{\#}};0^{[\frac{r}{2}]}\big)\Big)$, then $a'\succ 0^k$ by \eqref{yd-z}. Let
\[\begin{array}{c}a''=\big(\{\frac{r-q-m_t-1}{2},\frac{r-q-m_t-2k_t+1}{2}\};\dots; \{\frac{r-q-m_1-1}{2},\frac{r-q-m_1-2k_1+1}{2}\}\big).\end{array}\] 
By (i) of Lemma \ref{kt_of_zs}, $\big[E_{\gamma}^{K_1}:Z^{\#}|_{K_1}\big]_{K_1}\cdot\big[E_{\gamma}^{K_1}\otimes E_a^{K_1}: E_{\mathcal{Q}_1(y_{\delta})}^{K_1}]_{K_1}\neq 0$ only when $a$ satisfy $\sum a=\sum a'$. 

Notice that $a$ contains at most $k$'s non-zero entries. When $\big[E_{(\chi;\varepsilon)}^{K_2}:\pi_u|_{K_2}\big]_{K_2}\neq 0$, the $\chi$ contains at most $q$'s non-zero entries. If $\big[E_{(\chi;\varepsilon)}^{K_2}\otimes E_{\tiny \big(a,0^{[\frac{r}{2}]-k};\epsilon(a)\big)}^{K_2}: E_{(\mathcal{Q}_2(y_{\delta});\eta)}^{K_2}\big]_{K_2}>0$, notice that $\sum a''=\sum a'=\sum a$, by Lemma \ref{tensor_o}, one has $a=a''$. Therefore one can get only 
\[\begin{array}{c}\chi=(\delta+\{\frac{n-2k}{2},\frac{n-2k-2q+2}{2}\};0^{[\frac{r}{2}]-q})\end{array}\] 
can contribute to the left side of \eqref{mul=sum}. Moreover, by \eqref{kt_of_pi_u_od}, there is unique $\delta \in \{\pm\frac{1}{2}\}^q$ such that $(\delta+\{\frac{n-2k}{2},\frac{n-2k-2q+2}{2}\};0^{[\frac{r}{2}]-q})$ is a highest weight of the $K_2$-type of $\pi_u$. 
By the similar argument of Case (I) of Proposition \ref{prop_ktype}, only $\gamma=\zeta_{Z^{\#}}$ could produce $\big[E_{\gamma}^{K_1}:Z^{\#}|_{K_1}\big]_{K_1}\cdot\big[E_{\gamma}^{K_1}\otimes E_a^{K_1}: E_{\mathcal{Q}_1(y_{\delta})}^{K_1}]_{K_1}> 0$.

Fix the $\gamma=\zeta_{Z^{\#}}$, $a$ and $\delta$ as above, one can see that there is unique $\eta$ such that the left side of  \eqref{mul=sum} is nonzero. In this case, 
\[\begin{array}{rl} \big[(Z^{\#}\boxtimes \pi_u)\otimes \mathcal{E}_1:  E^K_{(y_{\delta};\eta)}\big]_{L\cap K}\vspace{0.8em}  = & \big[E_{\gamma}^{K_1}:Z^{\#}|_{K_1}\big]_{K_1}\cdot\big[E_{\gamma}^{K_1}\otimes E_a^{K_1}: E_{\mathcal{Q}_1(y_{\delta})}^{K_1}\big]_{K_1}\cdot  \\ & 
\big[E_{(\chi;\varepsilon)}^{K_2}:\pi_u|_{K_2}\big]_{K_2}\cdot \big[E_{(\chi;\varepsilon)}^{K_2}\otimes E_{\tiny \big(a,0^{[\frac{r}{2}]-k};\epsilon(a)\big)}^{K_2}: E_{(\mathcal{Q}_2(y_{\delta});\eta)}^{K_2}\big]_{K_2}.\end{array}\] 
One can check directly that every term on the right side equals to $1$, showing that the multiplicity is one.

This finishes the proof of this basic case and also the proof of the theorem. 
\end{proof}

\bigskip
\centerline{\scshape Acknowledgements}
 
We thank Professor C.-P. Dong for suggesting us classifying the Dirac series of $\mathrm{GL}(n,\mathbb{H})$ and helpful discussions about spin lowest $K$-types, and many valuable suggestions on the draft; thank Professor K.D. Wong for discussions on the Dirac series of $\mathrm{GL}(n,\mathbb{R})$ and helping modify the proof of the uniqueness of spin lowest $K$-type.
Zhang thanks Professor K.D. Wong for kindly supporting to visit Professor C.-P. Dong.  This work is supported by NSFC grant 11901491. We thank the referees for  reading the manuscript carefully and and helpful suggestions to improve the quality of this work.

\end{document}